\documentclass[ssy, preprint]{imsart}

\RequirePackage[OT1]{fontenc}
\RequirePackage{amsthm,amsmath,amssymb,amsfonts}
\RequirePackage[numbers]{natbib}
\RequirePackage[colorlinks,citecolor=blue,urlcolor=blue]{hyperref}

\arxiv{arXiv:1506.08657}

\usepackage{mathrsfs}
\startlocaldefs
\numberwithin{equation}{section}
\theoremstyle{plain}
\newtheorem{theorem}{Theorem}[section]
\newtheorem{lemma}{Lemma}[section]

\newtheorem{remark}{Remark}[section]
\newtheorem{corollary}{Corollary}[section]
\newcommand{\ul}{u_{\hspace{-0.15ex}{}_L}}
\newcommand{\tr}{\top}
\endlocaldefs

\numberwithin{table}{section}

\begin{document}

\begin{frontmatter}
\title{A Concentration Bound for Stochastic Approximation via Alekseev's Formula}
\runtitle{A Concentration Bound for Stochastic Approximation}

\begin{aug}
\author{\fnms{Gugan} \snm{Thoppe}\thanksref{t1,m1}
\ead[label=e1]{gugan@tcs.tifr.res.in}}
\and
\author{\fnms{Vivek} \snm{Borkar}\thanksref{t2,m2}
\ead[label=e2]{borkar.vs@gmail.com}}

\thankstext{t1}{Supported partially by an IBM PhD fellowship.}
\thankstext{t2}{Supported in part by a  J. C. Bose Fellowship and a grant `\emph{Approximation of High Dimensional Optimization and Control Problems}' from the Department of Science and Technology, Government of India.}
\runauthor{G. Thoppe \and V. Borkar}

\affiliation{TIFR, India\thanksmark{m1} and IITB, India\thanksmark{m2}}

\address{School of Tech. \& Comp. Sc.,\\
TIFR, Homi Bhabha Road,\\
Colaba, Mumbai, India.\\
Pincode: 400 005.\\
\printead{e1}}

\address{Dept. of Electrical Engineering, \\
IIT, Powai, Mumbai, India.\\
Pincode: 400 076.\\
\printead{e2}}
\end{aug}

\begin{abstract}
%
Given an ODE and its perturbation, the Alekseev formula expresses the solutions of the latter in terms related to the former. By exploiting this formula and a new concentration inequality for martingale-differences, we develop a novel approach for analyzing nonlinear Stochastic Approximation (SA). This approach is useful for studying a SA's behaviour close to a Locally Asymptotically Stable Equilibrium (LASE) of its limiting ODE; this LASE need not be the limiting ODE's only attractor. As an application, we obtain a new concentration bound for nonlinear SA. That is, given $\epsilon >0$ and that the current iterate is in a neighbourhood of a LASE, we provide an estimate for i.) the time required to hit the $\epsilon-$ball of this LASE, and ii.) the probability that after this time the iterates are indeed within this $\epsilon-$ball and stay there thereafter. The latter estimate can also be viewed as the `lock-in' probability. Compared to related results, our concentration bound is tighter and holds under significantly weaker assumptions. In particular, our bound applies even when the stepsizes are not square-summable. Despite the weaker hypothesis, we show that the celebrated Kushner-Clark lemma continues to hold.
\end{abstract}

\begin{keyword}[class=MSC]
\kwd[Primary ]{62L20}
\kwd[; secondary ]{93E25}
\kwd{60G42}
\kwd{34D10}
\end{keyword}

\begin{keyword}
\kwd{stochastic approximation}
\kwd{concentration bound}
\kwd{Alekseev's formula}
\kwd{perturbed ODE}
\kwd{sum of martingale-differences}
\kwd{concentration inequality}
\end{keyword}

\end{frontmatter}

\section{Introduction}









Stochastic Approximation (SA), first introduced in \cite{robbins1951stochastic}, refers to recursive methods that can be used to find optimal points or zeros of a function given only its noisy estimates. It is extremely popular in application areas such as adaptive signal processing, adaptive resource allocation, artificial intelligence, etc. Due to the stochastic nature of these methods, analysis of their convergence and convergence rates is challenging. For generic noise settings, the most powerful analysis tool has been the Ordinary Differential Equation (ODE) approach. Its idea is to show that the noise effects average out so that the asymptotic behavior of a SA method is determined by that of a suitable deterministic ODE, often referred to as the limiting ODE. For more details on the above, see \cite{borkar2008stochastic, benveniste2012adaptive, chen2002stochastic, duflo1997random, kushner2003stochastic, benaim1999dynamics}.


Here we analyze the behaviour of a nonlinear SA method close to a Locally Asymptotically Stable Equilibrium (LASE) of its limiting ODE\footnote{A LASE is an equilibrium that is Liapunov  stable in the following sense: given an $\eta > 0,$ there exists a $\delta > 0$ such that any trajectory of the ODE initiated within $\delta$ distance from this equilibrium remains within $\eta$ distance thereof; furthermore, there is an open neighbourhood such that any trajectory initiated therein converges to this equilibrium. This neighbourhood is called the domain of attraction.}. In particular, we obtain a novel \textit{concentration bound} for nonlinear SA methods. That is, given $\epsilon > 0$ and that the current iterate is in a neighbourhood within the domain of attraction of a LASE, we provide estimates on i.) the time required to hit the $\epsilon-$ball of this LASE, and ii.) the probability that after this time the iterates are indeed within this $\epsilon-$ball and remain there thereafter. Since staying within the $\epsilon-$ball of a LASE from some time on implies that the iterates will eventually converge to this equilibrium, the above probability estimate can also be viewed as an estimate on the so called \textit{`lock-in'} probability \cite{arthur1994increasing},  \cite[Chapter 4]{borkar2008stochastic}. Similar concentration bounds are already available in literature \cite{borkar2008stochastic, kamal2010convergence} in the context of generic attractors. Compared to these, our bound is stronger albeit restricted to the important special case of a LASE.  We achieve the tighter bound by using a finer analysis which strongly exploits the behaviour of ODE solutions near a LASE.


In case of multiple stable attractors, SA methods have a positive probability of convergence to any of them \cite{artur1983generalized}, \cite[Chapter 3]{arthur1994increasing}, \cite[Proposition 7.5]{benaim1999dynamics}. Thus one cannot expect willful convergence to a specific equilibrium except in some special cases; e.g., stochastic gradient schemes controlled by addition of slowly decreasing noise \cite{gelfand1991recursive}. Hence an important first step is to estimate the probability of convergence to an attractor given that the iterate is currently in its domain of attraction. The idea is that in such a situation the SA method will converge to the said attractor with high probability because the mean dynamics, as captured by the limiting ODE, favors it. This in fact is the basis for Arthur's models of increasing returns in economics \cite{arthur1994increasing}. To make the above qualitative (or ‘descriptive’) observation useful (or ‘prescriptive’) by giving it some predictive power, it is essential that those probabilities, the so called trapping or `lock-in' probabilities \cite{arthur1994increasing}, be estimated. This is what this work, and also \cite{borkar2008stochastic, kamal2010convergence}, attempt to do. This is also related in spirit to the extensively studied phenomenon of metastability in statistical physics wherein a statistical mechanical system spends a long time near a stable minimum of its governing energy function other than its global minimum or ‘ground state’ \cite{bovier2016metastability}; this would be the case, e.g., if we worked with constant stepsize SA methods instead of decreasing stepsizes.


In addition to our concentration bound, and more importantly, we provide here a novel approach for analysis of nonlinear SA. The main ingredient of our approach is Alekseev's formula \cite{alekseev1961estimate}; an English account can be found in \cite{brauer1966perturbations}. This formula extends the variation of constants formula \cite{lakshmikantham1998method} to nonlinear settings. That is, given two nonlinear ODEs where one can be treated as a perturbation of the other, Alekseev's formula gives an explicit expression for difference between the solutions of these two ODEs. The other ingredient of our approach is a novel concentration inequality for a sum of martingale-differences that we prove separately; see Theorem~\ref{thm:ConcMartingalesMultivariate} in Appendix. This result is a generalization of  \cite[Theorem 1.1]{liu2009exponential}.


As remarked above, concentration bounds in \cite[Chapter 4, Corollary 14]{borkar2008stochastic} and \cite[Theorem 12]{kamal2010convergence} are for generic attractors. But by taking the generic attractor to be a LASE, these results can be put in a form comparable to our result. It can be then seen that our bound is tighter and holds under significantly weaker assumptions on the stepsize and the noise sequence, but under a stronger regularity requirement (twice continuous differentiability) on the drift;  see Section~\ref{sec:Comparison}. In particular, our bound holds for a larger choice of stepsizes, e.g., $1/(n + 1)^{\mu},$ $\mu \in (0, 1],$ while the previous ones only apply for stepsizes that are square-summable. Despite the weaker hypothesis, we show that the celebrated Kushner-Clark  lemma \cite{kushner2012stochastic} continues to hold. All these happen mainly because the earlier two works use the weaker Gronwall inequality \cite[Corollary 1.1]{bainov1992integral}, while here we use the tighter Alekseev's formula to compare the SA trajectory to a suitable solution of its limiting ODE. In particular, Alekseev's formula allows us to better compare the two in the neighbourhood of a LASE on an infinite time interval.


Concentration bounds of similar flavor to our work have also been obtained recently in \cite[Theorem 2.2]{frikha2012concentration} and \cite[Corollary 2.9]{fathi2013transport}. But as shown in Section~\ref{sec:Comparison}, compared to our work and also to \cite{borkar2008stochastic, kamal2010convergence}, these recent results apply only to a restrictive class of SA methods and hold only under strong assumptions. In particular, results from \cite{frikha2012concentration, fathi2013transport} only apply to SA methods: i) whose form are special cases of the generic model that we handle; ii) whose limiting ODE has a unique, globally asymptotically stable equilibrium; and iii) that satisfy respectively the  assumption labelled HL (\cite{frikha2012concentration}) and  HLS$_\alpha$ (\cite{fathi2013transport}), amongst others. Both HL and its weaker variant HLS$_\alpha$ are strong assumptions; for e.g., they do not hold for the simple yet popular TD(0) method with linear function approximation from reinforcement learning \cite{dalal2018finite1}. Under these settings, their results give unconditional convergence rates. This is possible because of the unique equilibrium hypothesis; we shall discuss this issue further in Section~\ref{sec:Discussion}.

We now formally describe our setup and our key result in this paper. We consider the $d-$dimensional SA method
\begin{equation}
\label{eqn:SAIterates}
x_{n + 1} = x_{n} + a_n [h(x_n) +  M_{n + 1}], \; n \geq 0,
\end{equation}
where $\{a_n\}$ denotes a real valued stepsize sequence, $h : \mathbb{R}^d \to \mathbb{R}^d$ denotes a deterministic map, and $\{M_n\}$ denotes some $d-$dimensional noise sequence. The terms inside the square bracket represent the noisy measurement; i.e., one can only access the sum but not the individual terms separately.

Treating \eqref{eqn:SAIterates} as a noisy Euler scheme, the limiting ODE which this algorithm might be expected to track is
\begin{equation}
\label{eqn:LimitingODE}
\dot{x}(t) = h(x(t)).
\end{equation}
Let $x^*$ be a LASE of \eqref{eqn:LimitingODE} such that $Dh(x^*)$ is Hurwitz and let $B$ be a  bounded set containing $x^*$ and contained in the domain of attraction of $x^*.$ Let $\|\cdot\|$ denote the usual Euclidean norm for vectors and matrices. Let $\bar{x}(t)$ denote the continuous time version of \eqref{eqn:SAIterates} obtained via linear interpolation. That is, let $t_0 = 0$ and, for each $n \geq 0,$ set $t_{n + 1} = t_{n} + a_{n}$ and $\bar{x}(t_n) = x_n.$ For $t \in (t_{n}, t_{n + 1}),$ let
\begin{equation}
\label{eqn:LinearInterpolation}
\bar{x}(t) = \bar{x}(t_{n}) + \frac{(t - t_{n})}{a_{n}} [\bar{x}(t_{n + 1}) - \bar{x}(t_{n})].
\end{equation}
%

Let $n_0 \geq 0.$ Then given $\epsilon > 0$ and that the event $\{\bar{x}(t_{n_0}) \in B\}$ holds, our aim here is to obtain:
\begin{enumerate}
\item an estimate on the time $T,$ starting from $t_{n_0},$ that the SA method in \eqref{eqn:SAIterates} will take to hit the $\epsilon-$ball around $x^*;$ and

\item \label{itm:ConcBound} a lower bound on the probability that the SA method at time $t_{n_0} + T + 1$ is indeed inside the $\epsilon-$ball around $x^*$ and remains there thereafter, i.e.,
\begin{equation}
\label{eqn:SampleComplexity}
\Pr\{\|\bar{x}(t) - x^*\| \leq \epsilon \; \forall t \geq t_{n_0} + T + 1 \big| \bar{x}(t_{n_0}) \in B\}.
\end{equation}
\end{enumerate}
We assume the following throughout this paper.
{\renewcommand*\theenumi{$\pmb{A_\arabic{enumi}}$}
\begin{enumerate}
\item The map $h: \mathbb{R}^d \mapsto \mathbb{R}^d$ is $\mathcal{C}^2$ (twice continuously differentiable).

\item Stepsizes $\{a_n\}$ are strictly positive real numbers satisfying
\begin{equation}
\label{eqn:StepsizeAssumption}
\sum_{n} a_{n} = \infty,
\end{equation}
\[
\lim_{n \to \infty}a_{n} = 0,
\]
and
\begin{equation}
\label{eqn:StepsizeAssumption3}
\sup_{n} a_n \leq 1.
\end{equation}

\item The noise sequence $\{M_n\}$ is a $\mathbb{R}^d$ valued martingale-difference sequence with respect to the increasing family of $\sigma-$fields
\[
\mathcal{F}_{n} := \sigma(x_0, M_1, \ldots, M_n), n \geq 0.
\]
That is,
\begin{equation}
\label{eqn:NoiseAssumption1}
\mathbb{E}[M_{n + 1} | \mathcal{F}_{n}] = 0 \;  \text{a.s.}, n \geq 0.
\end{equation}
Furthermore, there exist continuous functions $c_1, c_2 : \mathbb{R}^d \to \mathbb{R}_{++}$ (strictly positive) such that
\begin{equation}
\label{eqn:NoiseAssumption2}
\Pr\{\|M_{n + 1}\| > u | \mathcal{F}_{n}\} \leq c_1(x_n) e^{-c_2(x_n) u}, \; n \geq 0,
\end{equation}
for all $u \geq \ul,$ where $\ul$ is some sufficiently large but fixed number.

\item There exist $r, r_0,\epsilon_0 > 0$ so that $r > r_0$ and, for $0 < \epsilon \leq \epsilon_0,$
\[
\{x \in \mathbb{R}^d: \|x - x^*\| \leq \epsilon \} \subseteq B \subseteq V^{r_0} \subset \mathcal{N}_{\epsilon_0}(V^{r_0}) \subseteq V^r \subset \text{dom}(V),
\]
where $V$ is some Liapunov function\footnote{Recall that a continuously differentiable function $V: \text{dom}(V) \subseteq \mathbb{R}^d \to \mathbb{R}$ is said to be a Liapunov function with respect to $x^*$ if $V(x^*) = 0$ and, for all $x \neq x^*,$ $V(x) > 0$ and $\nabla V(x) \cdot h(x) < 0.$ The existence of a Liapunov function near $x^*$ is guaranteed due to its asymptotic stability  by the converse Liapunov theorem \cite{krasovskii1963stability}. We may in fact choose $V$ so that $V(x) \rightarrow \infty$ as $x \rightarrow$ the boundary of dom$(V)$ (see \textit{ibid.}).} defined near $x^*,$ $\text{dom}(V)$ is the domain of the function $V,$
\[
V^{r_0} := \{x \in \text{dom}(V) : V(x) \leq r_0\},
\]
$V^r$ is defined similarly to $V^{r_0}$ with $r$ replacing $r_0,$ and
\[
\mathcal{N}_{\epsilon_0}(V^{r_0}) := \{x \in \mathbb{R}^d : \exists y \in V^{r_0} \text{ so that } \|x - y\| \leq \epsilon_0 \}.
\]
\end{enumerate}
}

\begin{remark}
Unlike most existing SA works \cite{borkar2008stochastic, kamal2010convergence, frikha2012concentration, fathi2013transport} etc., we do not require that the stepsize sequence $\{a_n\}$ satisfy the square summability condition, i.e., $\sum_{n} a_n^2 < \infty.$ Therefore, compared to these works, our analysis holds for larger choices of stepsizes; for e.g., $a_n = 1/(n + 1)^\mu,$ with $\mu \in (0, 1/2].$ As pointed out to us by an anonymous referee, similar slowly decaying stepsize sequences have appeared in \cite{polyak1992acceleration}; but there they appear only as part of the analysis for linear SA methods.
\end{remark}

\begin{remark}
\label{rem:AssumptionRestrictions}
We emphasize that the assumption $h$ is twice continuously differentiable globally is only for pedagogical convenience. Our results go through even if $h$ is twice continuously differentiable in some local neighbourhood of $x^*.$ Assumption \eqref{eqn:StepsizeAssumption3} is again for ease of notation. Our results with minor modifications can be obtained even without it.
\end{remark}

\noindent Let $Dh(x^*)$ be the Jacobian matrix of $h$ at $x^*$ and let $\lambda_1(x^*), \ldots, \lambda_d(x^*)$ denote its $d$ eigenvalues. Since $Dh(x^*)$ is Hurwitz, 
\begin{equation}
\label{eqn:Defn_LambdaMinx*}
\lambda_{\min}(x^*) := \min_i\{- \text{real}(\lambda_i(x^*))\}
\end{equation}
is a strictly positive number. Fix $\lambda^\prime$ such that $0 < \lambda^\prime < \lambda_{\min}(x^*)$ and $\kappa$ such that $0 < \kappa < 1.$ From \cite[Corollary 3.6, p71]{teschl2012ordinary}, there exists $\tilde{K} > 0$ such that\footnote{\label{fn:splCase} In the special case that $Dh(x^*)$ is symmetric and hence diagonalizable, $\tilde{K}$ and $\lambda^\prime$ can be chosen to be $1$ and $\lambda_{\min}(x^*),$ respectively.}
\begin{equation}
\label{eqn:BoundNormDhx*}
\|e^{Dh(x^*)t}\| \leq \tilde{K} e^{-\lambda^{\prime} t}, \; t \geq 0.
\end{equation}
Without loss of generality, we will assume that $\tilde{K} \geq 1.$ Let
\begin{equation}
\label{eqn:Defn_lambda}
\lambda = \left(\frac{1 - \kappa}{\tilde{K}^2} \right) \lambda^\prime.
\end{equation}
Clearly $\lambda < \lambda^\prime.$ The main result of this paper is the following.

\begin{theorem}
\label{thm:MainResult}
Suppose $\pmb{A}_1, \ldots, \pmb{A}_4$ hold. For each $n,$ let
\[
\beta_n := \max\limits_{n_0 \leq k \leq n - 1} \left[e^{-\lambda \sum_{i = k + 1}^{n - 1}a_i} \right] a_k.
\]
Let $\epsilon \in (0, \epsilon_0],$ where $\epsilon_0$ is as in $\pmb{A}_4.$ Then there exist constants $C_1, C_2 > 0$ and functions $g_1(\epsilon) = O\left(\log\left[\frac{1}{\epsilon}\right]\right)$ and $g_2(\epsilon) = O\left(\frac{1}{\epsilon}\right)$ so that whenever $T \geq g_1(\epsilon)$ and $n_0 \geq N,$ where $N$ is such that $1/a_n \geq g_2(\epsilon)$ $\forall n \geq N,$ the SA iterates of \eqref{eqn:SAIterates} satisfy:
\begin{multline*}
\Pr\{\|\bar{x}(t) - x^*\| \leq \epsilon \; \forall t \geq t_{n_0} + T + 1\big| \bar{x}(t_{n_0}) \in B\} \geq \\ 1 - \sum_{n = n_0}^{\infty} C_1 \exp\left(- \frac{C_2 \sqrt{\epsilon}}{\sqrt{a_n}}\right) - \sum_{n = n_0}^{\infty} C_1 \exp\left(-\frac{C_2 \min\{\epsilon, \epsilon^2\}}{\beta_{n}}\right).
\end{multline*}
Here the constants $C_1, C_2$ as also the hidden constants in $g_1, g_2$ depend only on $\lambda, d, r,$ and $\ul$.
\end{theorem}

The next result obtains order estimates for our concentration bound, for the common stepsize family $a_n = 1/(n + 1)^\mu, \mu \in (0,1].$

\begin{theorem}
\label{thm:OrderEst}
Let $a_n = 1/(n + 1)^\mu,$ $\mu \in (0, 1].$ With notations as in Theorem~\ref{thm:MainResult}, keeping everything else fixed and treating only $n_0$ as a variable,
\[
\Pr\{\|\bar{x}(t) - x^*\| \leq \epsilon \; \forall t \geq t_{n_0} + T + 1 \big| \bar{x}(t_{n_0}) \in B\} = 1 - O(n_0^{1 - \mu/2} e^{- C n_0^{\mu/2}})
\]
for some constant $C > 0.$ Here $O$ denotes the standard Big O order notation.
\end{theorem}
\begin{proof}
See Section~\ref{sec:OrderEstimates} in Appendix.
\end{proof}

Some notable aspects of Theorem~\ref{thm:MainResult} are as follows.
\begin{itemize}
\item It is a \textit{local} result, i.e., it gives a bound on the probability of convergence to a LASE if the iterates land up in its domain of attraction eventually. This is the so called lock-in probability \cite{arthur1994increasing}. In particular, $\{x^*\}$ need not be the only attractor of \eqref{eqn:LimitingODE}.

\item Letting $\mathcal{A}(n)$ denote complement of the event whose conditional probability appears in the statement, we have a bound of the form
\[
\Pr\{\mathcal{A}(n_0) | x_{n_0} \in B\} \leq c(n_0)
\]
for a suitably defined $c(n_0)$ satisfying $\sum_nc(n) < \infty$. Therefore,
\[
\sum_n\Pr\{\mathcal{A}(n) | x_n \in B\}I\{x_n \in B\} < \infty \ \mbox{a.s.},
\]
where $I$ denotes the indicator function. Consequently, by \cite[Corollary 5.29, p.\ 96]{breiman1992probability}, we have
\[
\sum_n I\{\mathcal{A}(n), x_n \in B\} < \infty \ \mbox{a.s.}
\]
In particular, this implies that $x_n \rightarrow x^*$ a.s. on the set $\{x_n \in B \; \mbox{i.o.}\}$. Thus we recover the celebrated Kushner-Clark lemma \cite{kushner2012stochastic} under the weaker hypothesis $a_n \to 0$ replacing the usual condition $\sum_n a_n^2 < \infty.$
\end{itemize}

There is also one key limitation to this result. The concentration bound is conditional on the event $\{\bar{x}(t_{n_0}) \in B\}.$ Thus in order to \textit{drive} the iterates to a prescribed equilibrium, one will need to separately ensure that the $n_0-$th iterate is indeed within the set $B.$ A related issue is to estimate the \textit{unconditional} probability of convergence to a prescribed equilibrium. This requires an estimate of the probability of reaching the domain of attraction of the prescribed equilibrium from a given starting point. We discuss this in Section \ref{sec:Discussion}. Three artificial ways to fix this are as follows. First, forcefully project the SA method back onto the set $B$ whenever the method leaves it; see \cite{dalal2018finite2} for recent advances in this direction. Second, pick an initial point within $B$ and scale-down the entire stepsize sequence so that $N$ given in Theorem~\ref{thm:MainResult} equals $0.$ Third, use additional additive, extraneous noise to `explore' the space, along with an oracle that tells you when you are in $B$. All these fixes are non-trivial as these require explicit a priori knowledge of the desired equilibrium (e.g., a global minimum of some function) because often (e.g., in engineering applications) that is precisely what the algorithm is expected to discover. Also, the set $B$ is often unavailable or hard to estimate well even when the desired equilibrium is known.  Separately, the second fix needs bounds on $g_2(\epsilon)$ which, as we shall see, depends on a priori unknown parameters such as the smallest eigenvalue of $Dh(x^*),$ amongst others. But whenever such estimates can be obtained, the following restatement of Theorem~\ref{thm:MainResult} may be useful.

\begin{corollary}
Suppose $\pmb{A}_1, \ldots, \pmb{A}_4$ hold. Let $\epsilon \in (0, \epsilon_0].$ Let $\beta_n,$ $T,$ and $g_2(\epsilon)$ be as in Theorem~\ref{thm:MainResult}. Suppose that $1/a_n  \geq g_2(\epsilon)$ $\forall n \geq 0$ and that $\bar{x}(0) = x_0$ lies in $B.$ Then the following relation holds:
\begin{multline*}
\Pr\{\|\bar{x}(t) - x^*\| \leq \epsilon \; \forall t \geq T + 1\} \geq \\ 1 - \sum_{n = 0}^{\infty} C_1 \exp\left(- \frac{C_2 \sqrt{\epsilon}}{\sqrt{a_n}}\right) - \sum_{n = 0}^{\infty} C_1 \exp\left(-\frac{C_2 \min\{\epsilon, \epsilon^2\}}{\beta_{n}}\right),
\end{multline*}
Here $C_1, C_2 > 0$ are constants as in Theorem~\ref{thm:MainResult}.
\end{corollary}

The rest of the paper is organized as follows. In the next section, we give a comparison of our main result with existing works. This section may be skipped at the first reading. In the following section, we do some preliminary computations and get an intermediate lower bound on \eqref{eqn:SampleComplexity} which will be easier to work with. We also give an overview of our proof technique for Theorem~\ref{thm:MainResult}. In Section~\ref{sec:AlternativeViaAlekseev}, we first give  Alekseev's formula. Using this, we then derive an alternative but equivalent expression for $\bar{x}(t)$ and, in particular, for $\bar{x}(t_n).$ In Section~\ref{sec:ErrorBound}, we use this alternative expression to obtain a bound on $\|\bar{x}(t_{n + 1})- x^*\|$ in terms of the noise sequence $\{M_n\}.$ In Section~\ref{sec:MainProof}, we finally prove our main result, i.e., Theorem~\ref{thm:MainResult}, via a series of Lemmas. This section needs a generalization of a concentration result from \cite{liu2009exponential}, which we prove separately as Theorem~\ref{thm:ConcMartingalesMultivariate} in the Appendix. We conclude with a brief discussion in Section~\ref{sec:Discussion}.

\section{Comparison with existing results}
\label{sec:Comparison}

We first compare our key result (Theorem~\ref{thm:MainResult}) with \cite[Chapter 4, Corollary 14]{borkar2008stochastic} and \cite[Theorem 12]{kamal2010convergence}. Those results give concentration bounds for nonlinear SA methods with respect to generic attractors. Replacing the generic attractor with a LASE, the results there are in a form directly comparable to our result. Let $B$ and $x^*$ be as in Theorem~\ref{thm:MainResult}. Those works first give an estimate on the additional time $T^\prime$ required to hit a suitably defined $\epsilon-$neighbourhood of $x^*$ starting from $B.$ This estimate is same in both those works. Then, an estimate is provided of the probability that after a passage of time $T^\prime$ the iterates are within the aforementioned $\epsilon-$neighbourhood and remain there thereafter, conditional on $\{x_{n_0} \in B\}.$ The estimate on $T^\prime$ is very loose; unlike our bound, it does not exploit the exponentially fast convergence rate of a ODE solution near a LASE. We compare their concentration bounds separately below.

There are two parts to \cite[Chapter 4, Corollary 14]{borkar2008stochastic}. The first part assumes
\begin{equation}
\label{eqn:hAssumptionBorkarKamal}
h\text{ is globally Lipschitz continuous},
\end{equation}
the stepsize sequence $\{a_n\}$ satisfies
\begin{equation}
\label{eqn:StepsizeAssumptionBorkarKamal}
\sum_{n = 0}^{\infty}a_{n} = \infty \text{ and } \sum_{n = 0}^{\infty}a_{n}^2 < \infty,
\end{equation}
and the martingale-difference noise sequence $\{M_n\}$ satisfies
\[
\mathbb{E}\left[\frac{\|M_{n+ 1}\|^2}{1 + \|x_n\|^2}\bigg| \mathcal{F}_n\right] \leq C_1
\]
for some constant $C_1 \geq 0.$ Under these assumptions, as $n_0 \to \infty$ with everything else fixed, it is shown that the concentration bound, defined above, is $1 - O(b_{n_0}/ \delta^2).$ Here $b_{n_0} := \sum_{n \geq n_0} a_{n}^2$ and $\delta$ is some constant depending on $\epsilon.$ For $a_n = 1/(n + 1)^\mu,$ $\mu \in (1/2, 1],$ it can be easily seen that the above concentration bound is $1 - O(1/n_0^{2\mu - 1}).$

In the second part of \cite[Chapter 4, Corollary 14]{borkar2008stochastic}, the assumptions on $h$ and $\{a_n\}$ are same as above. The difference is in the assumption on $\{M_n\}.$ It is assumed there that $\{M_n\}$ is a martingale-difference sequence such that
\begin{equation}
\label{eqn:NoiseAssumptionBorkar}
\frac{\|M_{n + 1}\|}{1 + \|x_n\|} \leq C_1, \;  n \geq 0, \text{ a.s.}
\end{equation}
for some constant $C_1 \geq 0.$ Under these assumptions, a concentration bound of $1 - O\left(\exp[- C_2\delta^2 /b_{n_0}]\right)$ is obtained. Here $C_2 > 0$ is another constant and $\delta, b_{n_0}$ are as above. For $a_n = 1/(n + 1)^\mu,$ $\mu \in (1/2, 1],$ this concentration bound is $1 - O(e^{-C_3 n_0^{2\mu - 1}})$ for some constant $C_3 > 0.$ Clearly the concentration bound in the second part is tighter. But the bounded noise assumption of \eqref{eqn:NoiseAssumptionBorkar} is restrictive and does not hold true in general (this setting is however very useful for many reinforcement learning problems).

The result of \cite[Theorem 12]{kamal2010convergence} significantly improves on this. In addition to \eqref{eqn:hAssumptionBorkarKamal} and \eqref{eqn:StepsizeAssumptionBorkarKamal}, it is only assumed there that $\{M_n\}$ is a martingale-difference sequence and that, for some constants $C_1, C_2 > 0,$
\[
\Pr\left\{\frac{\|M_{n + 1}\|}{1 + \|x_n\|}  > u \bigg| \mathcal{F}_n\right\} \leq C_1 \exp(-C_2 u), \; n \geq 0,
\]
for all sufficiently large $u.$ Under these assumptions, it is shown there that \eqref{eqn:SampleComplexity}, with $T$ replaced by $T^\prime,$ is $1 - O\left(\exp[-C_3\delta^{2/3}/\sqrt[4]{b_{n_0}}]\right)$ for some constant $C_3 > 0$ and $\delta, b_{n_0}$ as above. For $a_n = 1/(n + 1)^\mu,$ $\mu \in (1/2, 1],$ it is easy to see that this bound translates to $1 - O(e^{-C_4 n_0^{\mu/2 - 1/4}})$ for some constant $C_4 > 0.$ Compared to the second part of \cite[Chapter 4, Corollary 14]{borkar2008stochastic}, this bound is weaker but it also has a similar exponential behaviour in $n_0.$

Our result, i.e., Theorem~\ref{thm:MainResult} of this paper, significantly improves upon the above two results. First, we do not need that the stepsize sequence satisfy the square summability condition $\sum_{n = 0}^{\infty}a_n^2 < \infty;$ instead, we only require that $a_n \to 0$ (in addition to $\sum_{n = 0}^{\infty} a_n = \infty).$ Thus, our result holds even for stepsizes such as $a_n = 1/(n + 1)^\mu,$ with $\mu \in (0, 1/2],$ while the previous two do not. Second, we only require that the noise sequence $\{M_n\}$ satisfy $\pmb{A}_3.$ This is weaker than the assumption on $\{M_n\}$ made in \cite[Theorem 12]{kamal2010convergence} (and hence in the second part of \cite[Chapter 4, Corollary 14]{borkar2008stochastic}). Third, despite the weaker assumptions, a direct comparison of our concentration bound for $a_{n} = 1/(n + 1)^\mu,$ $\mu \in (1/2, 1]$ (see Theorem~\ref{thm:OrderEst}), shows that our bound betters that in \cite[Theorem 12]{kamal2010convergence} for all $\mu \in (1/2, 1]$ and the one in the second part of \cite[Chapter 4, Corollary 14]{borkar2008stochastic} for $\mu \in (1/2, 2/3).$ Lastly, by exploiting the exponential convergence of a ODE solution near its attractor, we obtain tighter estimates for the time $T$ required to hit the $\epsilon-$ball around $x^*$ starting from the neighbourhood $B.$ We do, however, require a stronger regularity of the function $h$, viz., twice continuous differentiability, at least locally near the equilibrium. A brief summary of the above comparison is given in Table~\ref{tab:compAnalysis}.

\begin{table}[t]
\centering
\hspace{50ex}\caption{\label{tab:compAnalysis} Comparison of Theorem~\ref{thm:MainResult} with relevant results from literature.}
\begin{tabular}{ c | c | c | c}
\hline
Id & \begin{tabular}{@{}c@{}}Assumption \\ on $h:$\end{tabular} & \begin{tabular}{@{}c@{}} Assumption \\ on $\{M_n\}:$ \\  \eqref{eqn:NoiseAssumption1} and \end{tabular} & \begin{tabular}{@{}c@{}} Concentration \\ Bound: \end{tabular}\\[1ex]
\hline
& & &\\
$B_1$ & \begin{tabular}{@{}l@{}} Lipschitz \\ continuous \end{tabular}  & $\mathbb{E}\left[\dfrac{\|M_{n + 1}\|^2}{1 + \|x_n\|^2} \bigg| \mathcal{F}_n\right]\leq C_1 $ & $1 - O\left(\frac{1}{n_0^{2\mu  - 1}}\right)$\\[8ex]
$B_2$ & \begin{tabular}{@{}l@{}} Lipschitz \\ continuous \end{tabular} & $\dfrac{\|M_{n + 1}\|}{1 + \|x_n\|} \leq C_1$ & $1 - O\left(e^{- C_3 n_0^{2 \mu - 1}}\right)$ \\[8ex]
$K$ & \begin{tabular}{@{}l@{}} Lipschitz \\ continuous \end{tabular} & \begin{tabular}{@{}l@{}} $\Pr\left\{\dfrac{\|M_{n + 1}\|}{1 + \|x_n\|} > u \bigg| \mathcal{F}_{n}\right\}$ \\[2ex] $\leq C_1 e^{-C_2 u}$ \\[0.25ex] $\forall u$ sufficiently large \end{tabular} & $1 - O\left(e^{- C_4 n_0^{(2\mu - 1)/4}}\right)$ \\[10ex]
$*$ & $\mathcal{C}^2$ & \begin{tabular}{@{}l@{}}  $\Pr\{\|M_{n + 1}\| > u | \mathcal{F}_{n}\}$ \\[1ex] $\leq c_1(x_n) e^{-c_2(x_n) u}$ \\[0.25ex] $\forall u$ sufficiently large \end{tabular} & \begin{tabular}{@{}l@{}}  $1 - \tilde{O}(e^{-C n_0^{\mu/2}})$
\end{tabular}

\\
\hline
\end{tabular}\\[2ex]

$B_1$ and $B_2$ are respectively the first and second parts of \cite[Chapter 4, Corollary 14]{borkar2008stochastic}, $K$ is \cite[Theorem 12]{kamal2010convergence}, and $*$ is Theorem~\ref{thm:MainResult} from this paper. Each $C_i, C$  denotes a positive constant, $O$ is the Big O  notation, while $\tilde{O}$ is the Big O notation with polynomial terms hidden. The concentration bounds are obtained assuming $a_n = 1/(n + 1)^\mu$ with $\mu \in (1/2, 1]$ for the first three bounds, while $\mu \in (0, 1]$ for the last one.
\end{table}

The main reason why we obtain a tighter concentration bound in comparison to \cite[Chapter 4, Corollary 14]{borkar2008stochastic} and \cite[Theorem 12]{kamal2010convergence} is the following. In \cite{borkar2008stochastic,kamal2010convergence}, the analysis boils down to showing $\sum_{k = n_i}^{n} a_{k} M_{k + 1}$ is small in magnitude with high probability for all appropriately large $n_i$ and $n.$ In contrast, in the proof of our result, we only need to show that a term similar to $\sum_{k = n_0}^{n}e^{-\lambda [\sum_{i = k + 1}^{n} a_i]} a_{k} M_{k + 1},$ where $\lambda$ is as in \eqref{eqn:Defn_lambda}, is small for all large $n$ with high probability. This happens mainly due to the use of Alekseev's formula \cite{alekseev1961estimate} which allows us to exploit the local stability of the ODE near an attractor. Further, to show that the term similar to $\sum_{k = n_0}^{n}e^{-\lambda [\sum_{i = k + 1}^{n} a_i]} a_{k} M_{k + 1}$ is small, we make use of the concentration inequality given in Theorem~\ref{thm:ConcMartingalesMultivariate} in place of the Azuma-Hoeffding inequality as in \cite[Corollary 14]{borkar2008stochastic} and \cite[Theorem 12]{kamal2010convergence}.

Concentration bounds related to our work are also given in \cite[Theorem 2.2]{frikha2012concentration} and \cite[Corollary 2.9]{fathi2013transport}. However we discuss these separately since, as mentioned before, these results only apply to a restrictive class of SA methods and hold only under strong assumptions. Specifically, the SA algorithms considered there are of the form
\begin{equation}
\label{eqn:SAIteratesFrikha}
x_{n + 1} = x_{n} + a_{n} H(x_n, Y_{n + 1}),
\end{equation}
where: i.) $H: \mathbb{R}^d \times \mathbb{R}^d \to \mathbb{R}^d$ is a deterministic map satisfying the assumption labelled HL in \cite{frikha2012concentration} and  HLS$_\alpha$ in \cite{fathi2013transport}, amongst others; ii.) $\{a_n\}$ is a real valued step size sequence satisfying \eqref{eqn:StepsizeAssumptionBorkarKamal}; and iii.) $\{Y_{n}\}$ is a $\mathbb{R}^d$ valued sequence of IID random variables satisfying the Gaussian concentration property, i.e., there exist some $\alpha > 0$ so that for every $1-$Lipschitz function $f : \mathbb{R}^d \to \mathbb{R},$
\[
\mathbb{E}[\exp(\bar{\lambda} f(Y_1))] \leq \exp\left(\bar{\lambda} \mathbb{E}[f(Y_1)] + \frac{\alpha \bar{\lambda}^2}{4}\right), \; \bar{\lambda} \geq 0.
\]

By adding and subtracting $\mathbb{E}[H(x_n, Y_{n + 1})|\mathcal{F}_n],$ where $\mathcal{F}_n$ is the $\sigma-$field $\sigma(x_0, Y_1, \ldots, Y_n),$ it is easy to see that \eqref{eqn:SAIteratesFrikha} can be rewritten as in \eqref{eqn:SAIterates}; thereby showing that the above SA model is a special case of our SA model. Further, they assume that the limiting ODE has only one unique solution $x^*$; again a substantial simplification of the setup we consider. Both HL and  HLS$_\alpha$ relating to the growth of $H$ with respect to the second parameter are strong assumptions; in that they do not hold for the simple yet popular TD(0) method with linear function approximation \cite{dalal2018finite1}. As already discussed before, the square summability assumption on the stepsize is again stronger than ours. Lastly, since the Gaussian concentration property for $\{Y_n\}$ needs to hold true for every $1-$Lipshitz function $f,$ this requirement is also restrictive when compared with \eqref{eqn:NoiseAssumption2}. Under the restrictive settings and assumptions mentioned above, \cite{frikha2012concentration,fathi2013transport} obtain an upper bound on
\begin{equation}
\label{eqn:ConcentrationBoundFrikha}
\Pr\{\|x_n - x^*\| > \epsilon + \delta_n\}, \; n \geq 0,
\end{equation}
where $\delta_n := \mathbb{E}[\|x_n - x^*\|].$  There is a separate bound on $\delta_n$ which must be combined with the above to get the overall error bound.  Note that  \eqref{eqn:ConcentrationBoundFrikha} is unconditional, which is possible because of the strong assumptions and since there is a unique globally asymptotically stable equilibrium. Overall,  our concentration bound is of a similar flavor to the ones obtained in \cite{frikha2012concentration, fathi2013transport}.

\section{Preliminary computations}
\label{sec:Overview}

Henceforth, for $u_0 \in \mathbb{R}^d$ and $s \geq 0,$ we shall use $x(t, s, u_0),$ $t \geq s,$ to denote the solution of \eqref{eqn:LimitingODE} satisfying $x(s, s, u_0) = u_0.$ Suppose for the time being that $\bar{x}(t_{n_0}) \in B.$ Since from $\pmb{A_4},$ $B \subseteq V^{r_0}$ and $V$ is a Liapunov function, we have $x(t, t_{n_0}, \bar{x}(t_{n_0})) \in V^{r_0}$ for all $t \geq t_{n_0}.$ Further, if we wait long enough, then $x(t, t_{n_0}, \bar{x}(t_{n_0}))$ will reach a sufficiently close enough neighbourhood of $x^*$ and remain in it thereafter. Our idea to prove Theorem~\ref{thm:MainResult} is to show that with very high probability, conditional on $\{\bar{x}(t_{n_0}) \in B\},$ $\|\bar{x}(t) - x(t, t_{n_0}, \bar{x}(t_{n_0}))\|$ is small for all $t \geq t_{n_0}.$ Note that $\bar{x}(t)$ and $x(t, t_{n_0}, \bar{x}(t_{n_0}))$ start from the same point $\bar{x}(t_{n_0})$ at time $t = t_{n_0}.$ We elaborate more on our idea at the end of this section. But we first introduce some notations and come up with an intermediate lower bound on \eqref{eqn:SampleComplexity} which will be much easier to work with.

Fix some sufficiently large $n_0, T.$ We shall elaborate later on how large they ought to be. Pick $n_1 \equiv n_1(n_0)$ such that
\begin{equation}
\label{eqn:TReln1}
T \leq t_{n_1 + 1} - t_{n_0} = \sum_{n = n_0}^{n_1} a_n \leq T + 1.
\end{equation}
This can be done because \eqref{eqn:StepsizeAssumption} and \eqref{eqn:StepsizeAssumption3} hold. Let
\begin{equation}
\label{eqn:Defn_Rho}
\rho_{n + 1} := \sup_{t \in [t_n, t_{n + 1}]} \|\bar{x}(t) - x(t, t_{n_0}, \bar{x}(t_{n_0}))\|,
\end{equation}
\begin{equation}
\label{eqn:Defn_Rho*}
\rho_{n + 1}^{*} := \sup_{t \in [t_n, t_{n + 1}]} \|\bar{x}(t) - x^*\|,
\end{equation}
and
\begin{equation}
\label{eqn:Defn_G_n}
G_n := \{ \bar{x}(t) \in V^r \; \forall t \in [t_{n_0}, t_{n}]\}.
\end{equation}
Note that $G_n$ is an event and $G_{n_0} = \{\bar{x}(t_{n_0}) \in V^r\}.$

The desired intermediate lower bound on \eqref{eqn:SampleComplexity} is given below.

\begin{lemma}
For $n_0, n_1$ and $T$ that satisfy \eqref{eqn:TReln1},
\begin{multline}
\label{eqn:IntConcBound}
\Pr\{\|\bar{x}(t) - x^*\| \leq \epsilon \; \forall t \geq t_{n_0} + T + 1 \big| \bar{x}(t_{n_0}) \in B\}  \geq \\
1- \Pr\left\{\bigcup_{n = n_0}^{n_1} \{G_n,  \rho_{n + 1} > \epsilon\} \cup \bigcup_{n = n_1 + 1}^{\infty} \{G_n, \rho_{n + 1}^* > \epsilon\} \bigg| \bar{x}(t_{n_0}) \in B\right\}.
\end{multline}
\end{lemma}
\begin{proof}
Using \eqref{eqn:TReln1}, it follows that \eqref{eqn:SampleComplexity} satisfies the following relation.
\begin{eqnarray}
& & \Pr\{\|\bar{x}(t) - x^*\| \leq \epsilon \; \forall t \geq t_{n_0} + T + 1 \big| \bar{x}(t_{n_0}) \in B\} \nonumber \\
& \geq & \Pr\{\|\bar{x}(t) - x^*\| \leq \epsilon \; \forall t \geq t_{n_1 + 1} \big| \bar{x}(t_{n_0}) \in B\} \nonumber \\
& =  & \Pr\left\{\bigcap_{n = n_1 + 1}^{\infty} \{\rho_{n + 1}^{*} \leq \epsilon\} \bigg| \bar{x}(t_{n_0}) \in B\right\}  \nonumber \\
& = & 1 - \Pr\left\{\bigcup_{n = n_1 + 1}^{\infty} \{\rho_{n + 1}^{*} > \epsilon\} \bigg| \bar{x}(t_{n_0}) \in B\right\}  \nonumber \\
& = & 1 - \Pr\left\{\bar{x}(t_{n_0}) \in B, \bigcup_{n = n_1 + 1}^{\infty} \{\rho_{n + 1}^{*} > \epsilon\} \bigg| \bar{x}(t_{n_0}) \in B\right\} . \label{eqn:ConcBoundAlt3}
\end{eqnarray}
In the remaining part of this proof, we obtain a superset of the event in the second term in \eqref{eqn:ConcBoundAlt3}. This will help us prove the desired result.

For any event $E,$ let $E^c$ denote its complement. Then between any two events $E_1$ and $E_2,$ the following relation is easy to see.
\[
E_1 = (E_2 \cap E_1) \cup (E_2^c \cap E_1) \subseteq E_2 \cup (E_2^c \cap E_1).
\]
Using this, it follows that
\begin{multline*}
\bigcup_{n = n_1 + 1}^{\infty} \{\rho_{n + 1}^{*} > \epsilon\}  \subseteq \left\{\left[\sup\limits_{n_0 \leq n \leq n_1} \rho_{n + 1}\right] > \epsilon_0 \right\} \\ \cup \left\{\left[\sup\limits_{n_0 \leq n \leq n_1} \rho_{n + 1}\right] \leq \epsilon_0,  \bigcup_{n = n_1 + 1}^{\infty} \{\rho_{n + 1}^{*} > \epsilon\} \right\}
\end{multline*}
and hence
\begin{multline*}
\left\{\bar{x}(t_{n_0}) \in B, \bigcup_{n = n_1 + 1}^{\infty} \{\rho_{n + 1}^{*} > \epsilon\} \right\}  \subseteq \left\{\bar{x}(t_{n_0}) \in B, \left[\sup\limits_{n_0 \leq n \leq n_1} \rho_{n + 1}\right] > \epsilon_0 \right\} \\ \cup \left\{\bar{x}(t_{n_0}) \in B, \left[\sup\limits_{n_0 \leq n \leq n_1} \rho_{n + 1}\right] \leq \epsilon_0,  \bigcup_{n = n_1 + 1}^{\infty} \{\rho_{n + 1}^{*} > \epsilon\} \right\}.
\end{multline*}
Recall that, on the event $\{\bar{x}(t_{n_0}) \in B\},$ $x(t, t_{n_0}, \bar{x}(t_{n_0})) \in V^{r_0}$ for all $ t \geq t_{n_0}.$ Combining this with the assumption from $\pmb{A_4}$ that $\mathcal{N}_{\epsilon_0}(V^{r_0}) \subseteq V^{r},$ we get
\begin{eqnarray*}
& & \left\{\bar{x}(t_{n_0}) \in B, \left[\sup\limits_{n_0 \leq n \leq n_1} \rho_{n + 1}\right] > \epsilon_0 \right\} \\
& = & \{\bar{x}(t_{n_0}) \in B, \rho_{n_0 + 1} > \epsilon_0\} \\
& & \cup \bigcup_{n = n_0 + 1}^{n_1} \left\{\bar{x}(t_{n_0}) \in B, \left[\sup\limits_{n_0 \leq k < n} \rho_{k + 1}\right] \leq \epsilon_0,  \rho_{n + 1} > \epsilon_0\right\}\\
& \subseteq & \bigcup_{n = n_0}^{n_1} \{G_n,  \rho_{n + 1} > \epsilon_0\}\\
& \subseteq & \bigcup_{n = n_0}^{n_1} \{G_n,  \rho_{n + 1} > \epsilon\},
\end{eqnarray*}
where the last relation follows because $\epsilon_0 \geq \epsilon$ (see $\pmb{A_4}$). Arguing similarly and using the assumption from $\pmb{A_4}$ that $\{x \in \mathbb{R}^d : \|x - x^*\| \leq \epsilon\} \subseteq V^r,$ we get
\begin{eqnarray*}
& & \left\{\bar{x}(t_{n_0}) \in B, \left[\sup\limits_{n_0 \leq n \leq n_1} \rho_{n + 1}\right] \leq \epsilon_0,  \bigcup_{n = n_1 + 1}^{\infty} \{\rho_{n + 1}^{*} > \epsilon\} \right\}\\
& \subseteq & \left\{G_{n_1 + 1}, \bigcup_{n = n_1 + 1}^{\infty} \{\rho_{n + 1}^{*} > \epsilon\} \right\}\\
& \subseteq & \bigcup_{n = n_1 + 1}^{\infty} \{G_n, \rho_{n + 1}^* > \epsilon\}.
\end{eqnarray*}
Putting the above discussions together, we have
\begin{eqnarray*}
& & \left\{\bar{x}(t_{n_0}) \in B, \bigcup_{n = n_1 + 1}^{\infty} \{\rho_{n + 1}^{*} > \epsilon\} \right\}\\
& \subseteq & \bigcup_{n = n_0}^{n_1} \{G_n,  \rho_{n + 1} > \epsilon\} \cup \bigcup_{n =  n_1 + 1}^{\infty} \{G_n, \rho_{n + 1}^* > \epsilon\},
\end{eqnarray*}
which in combination with \eqref{eqn:ConcBoundAlt3}, gives the desired result.
\end{proof}

We now elaborate on our technique to prove Theorem~\ref{thm:MainResult} and the usefulness of \eqref{eqn:IntConcBound} for the same. First note that to obtain a lower bound on \eqref{eqn:SampleComplexity}, it suffices to obtain an upper bound on the second term on the RHS of \eqref{eqn:IntConcBound}. Indeed, this is what we do. This is also easier because we now only need to obtain bounds on $\rho_{n + 1}$ and $\rho_{n + 1}^*$ on the event $G_n.$ This has been done in Lemmas~\ref{lem:Bound_rho_n+1} and \ref{lem:Bound_rho_n+1^*} in Section~\ref{sec:ErrorBound}, where $S_n$ is an appropriate sum of martingale-differences. To show that the terms on the RHS there are small, we use the concentration inequality in Theorem~\ref{thm:ConcMartingalesMultivariate} and the assumption in \eqref{eqn:NoiseAssumption2}. In the next section, we describe Alekseev's formula and use it to give an alternative expression for $\bar{x}(t_n).$ This will be very useful for  proving Lemmas~\ref{lem:Bound_rho_n+1} and \ref{lem:Bound_rho_n+1^*}.

\section{Alekseev's formula and an alternative expression for $\bar{x}(t_n)$}
\label{sec:AlternativeViaAlekseev}

Alekseev's formula given below provides a recipe to compare two nonlinear systems of differential equations. This is a generalization of variation of constants formula.

\begin{theorem}[Alekseev's formula, \cite{alekseev1961estimate}]
\label{thm:AlekseevStatement}
Consider a differential equation
\[
\dot{u}(t) = f(t, u(t)), \; t \geq 0,
\]
and its perturbation
\[
\dot{p}(t) = f(t, p(t)) + g(t, p(t)), \; t \geq 0,
\]
where $f, g: \mathbb{R} \times \mathbb{R}^d \to \mathbb{R}^d,$ $f$ is continuously differentiable everywhere, and $g$ is continuous everywhere. Let $u(t, t_0, p_0)$ and $p(t, t_0, p_0)$ denote respectively the solutions to the above nonlinear systems for $t \geq t_0$ satisfying $u(t_0, t_0, p_0) = p(t_0, t_0, p_0) = p_0.$ Then,
\[
p(t, t_0, p_0) = u(t, t_0, p_0) + \int_{t_0}^{t} \Phi(t, s, p(s, t_0, p_0)) \; g(s, p(s, t_0, p_0)) ds, \;  t \geq t_0,
\]
where $\Phi(t, s, u_0),$ for $u_0 \in \mathbb{R}^d,$ is the fundamental matrix of the linear system
\begin{equation}
\label{eqn:VariationalSystemAlekseevLemma}
\dot{v}(t) = \frac{\partial f}{\partial u} (t, u(t, s, u_0)) \; v(t), \; t\geq s,
\end{equation}
with $\Phi(s, s, u_0) = \mathbb{I}_d,$ the $d-$dimensional identity matrix.
\end{theorem}

See \cite[Lemma 3]{brauer1966perturbations} for an English version of the original proof for the above result. We now use this result to compare $\bar{x}(t)$ with $x(t, t_{n_0}, \bar{x}(t_{n_0})).$
Using \eqref{eqn:SAIterates} and since $t_{k + 1} - t_k = a_k,$  $k \geq 0,$ we have, for any $n \geq n_0,$
\begin{eqnarray*}
\bar{x}(t_{n + 1}) & = & \bar{x}(t_{n_0}) + \sum_{k = n_0}^{n} a_k h(\bar{x}(t_k)) + \sum_{k = n_0}^{n} a_k M_{k + 1} \\
& = & \bar{x}(t_{n_0}) + \sum_{k = n_0}^{n} \int_{t_k}^{t_{k + 1}} h(\bar{x}(t_k)) ds + \sum_{k = n_0}^{n} \int_{t_k}^{t_{k + 1}} M_{k + 1} ds.
\end{eqnarray*}
For $k \geq n_0$ and $s \in [t_k, t_{k + 1}],$ define
\begin{equation}
\label{eqn:DiscretizationNoise}
\zeta_1(s) = h(\bar{x}(t_k)) - h(\bar{x}(s))
\end{equation}
and
\begin{equation}
\label{eqn:MartingaleNoise}
\zeta_2(s) = M_{k + 1}.
\end{equation}
Then it is easy to see that for $n \geq n_0$
\[
\bar{x}(t_{n + 1}) = \bar{x}(t_{n_0}) + \int_{t_{n_0}}^{t_{n + 1}} h(\bar{x}(s)) ds + \int_{t_{n_0}}^{t_{n + 1}} \zeta_1(s) ds + \int_{t_{n_0}}^{t_{n + 1}} \zeta_2(s) ds
\]
and in fact for $t \geq t_{n_0}$
\begin{equation}
\label{eqn:PerturbedODE}
\bar{x}(t) = \bar{x}(t_{n_0}) + \int_{t_{n_0}}^{t}h(\bar{x}(s))ds + \int_{t_{n_0}}^{t}\zeta_{1}(s) ds + \int_{t_{n_0}}^{t} \zeta_{2}(s)ds.
\end{equation}

Think of \eqref{eqn:LimitingODE} as the unperturbed ODE and \eqref{eqn:PerturbedODE} as its perturbation. The perturbation term at time $t$ is of course $\zeta_1(t) + \zeta_2(t),$ which is piecewise continuous in $t.$ The same proof that was used to prove Theorem~\ref{thm:AlekseevStatement} also holds in this context. Hence, using Alekseev's formula, we get
\begin{multline}
\label{eqn:AlekseevSAIterates}
\bar{x}(t) = x(t, t_{n_0}, \bar{x}(t_{n_0})) + \int_{t_{n_0}}^{t} \Phi(t, s, \bar{x}(s)) \zeta_1(s)ds \\ + \;  \int_{t_{n_0}}^{t} \Phi(t, s, \bar{x}(s)) \zeta_2(s)ds,
\end{multline}
where $\Phi(t, s, u_0),$ for any $u_0 \in \mathbb{R}^d,$ is the fundamental matrix of the non-autonomous linearized system
\begin{equation}
\label{eqn:VariationalEquation}
\dot{y}(t) = Dh(x(t, s, u_0)) y(t), \ t \geq s,
\end{equation}
with $\Phi(s, s, u_0) = \mathbb{I}_d.$ Here $Dh(x(t, s, u_0))$ is the Jacobian matrix of $h$ along the solution trajectory $x(t, s, u_0).$

Using \eqref{eqn:DiscretizationNoise}, \eqref{eqn:MartingaleNoise}, and \eqref{eqn:AlekseevSAIterates}, the following result is now immediate. This gives the desired alternative expression for $\bar{x}(t_n).$

\begin{theorem}
\label{thm:AlternativeExp_xnViaAlekseev}
Let $\bar{x}(t)$ be as in \eqref{eqn:LinearInterpolation}. Then
\begin{equation}
\label{eqn:AlternativeExp_xnViaAlekseev}
\bar{x}(t_{n}) = x(t_n, t_{n_0}, \bar{x}(t_{n_0})) + W_n + S_{n} + (\tilde{S}_{n} - S_{n}),
\end{equation}
where
\begin{equation}
\label{eqn:Defn_Wn}
W_{n} := \sum_{k = n_0}^{n - 1} \int_{t_k}^{t_{k + 1}} \Phi(t_{n}, s, \bar{x}(s)) [h(\bar{x}(t_k)) - h(\bar{x}(s))] ds,
\end{equation}
\begin{equation}
\label{eqn:Defn_tildeSn}
\tilde{S}_{n} := \sum_{k = n_0}^{n - 1} \left[ \int_{t_k}^{t_{k + 1}} \Phi(t_{n}, s, \bar{x}(s)) ds \right] M_{k + 1},
\end{equation}
and
\begin{equation}
\label{eqn:Defn_Sn}
S_{n} := \sum_{k = n_0}^{n - 1} \left[\int_{t_k}^{t_{k + 1}} \Phi(t_{n}, s, \bar{x}(t_k)) ds \right] M_{k + 1},
\end{equation}
with $\Phi(t_{n}, s, \bar{x}(s))$ being the fundamental matrix of \eqref{eqn:VariationalEquation} with $u_0 = \bar{x}(s).$
\end{theorem}

\begin{remark}
Note that $\{S_n\}$ is a sum of martingale-differences with respect to $\{\mathcal{F}_{n}\},$ while $\tilde{S}_{n}$ is not. We shall exploit this later while proving Theorem~\ref{thm:MainResult}.
\end{remark}

\section{Bound on $\rho_{n + 1}, \rho_{n + 1}^*$ on $G_n$}
\label{sec:ErrorBound}
Our aim here is to obtain a bound on $\rho_{n + 1}, \rho_{n + 1}^*$ on the event $G_n.$ This is given in Lemmas~\ref{lem:Bound_rho_n+1} and \ref{lem:Bound_rho_n+1^*}. We shall use this in Section~\ref{sec:MainProof} to obtain a bound on the second term on the RHS of \eqref{eqn:IntConcBound} and hence on \eqref{eqn:SampleComplexity}. The proof of the above mentioned results require some supplementary lemmas which we prove first. Across these lemmas, we shall repeatedly use the linear ODE
\begin{equation}
\label{eqn:LinearizedSystem}
\dot{z}(t) = Dh(x^*) z(t).
\end{equation}
This is the linearization of \eqref{eqn:LimitingODE} near $x^*.$ We shall also use $r$ as in $\pmb{A_4}$ and
\begin{equation}
\label{eqn:Defn_R}
R:= \sup_{x \in V^r} \|x - x^*\|.
\end{equation}
Separately, from $\pmb{A_1},$ recall that $h \in \mathcal{C}^2.$ Hence it follows that $h$ and $Dh$ are Lipschitz continuous over the compact set $V^r.$ Let $L_h$ and $L_D,$ respectively, denote the associated Lipschitz constants.

\begin{lemma}
\label{lem:CloseSolutionsBound}
Let $\lambda$ be as in \eqref{eqn:Defn_lambda}. Let $u_0, u_1$ be arbitrary points in $V^r$ and $s$ be an arbitrary positive real number. Then for $t \geq s,$
\[
\|x(t, s, u_0) - x(t, s, u_1)\| \leq K_1 \|u_0 - u_1\| e^{-\lambda(t - s)},
\]
where $K_1 \geq 0$ is some constant.
\end{lemma}
\begin{proof}
We first prove the following claim.

\textbf{Claim (i)}
There exists $r^\prime$ satisfying $0 < r^\prime < r$ with the following property. For any arbitrary $u_0, u_1 \in V^{r^\prime}$ and any $s \geq 0,$
\[
\|x(t, s, u_0) - x(t, s, u_1)\| \leq K_1^\prime \|u_0 - u_1\| e^{-\lambda(t - s)},
\]
where $K_1^\prime \geq 0$ is some constant.

Let
\begin{equation}
\label{eqn:Defn_P}
P := \int_{0}^{\infty} e^{[Dh(x^*)^\tr]t} e^{ [Dh(x^*)]t} dt,
\end{equation}
where $\tr$ denotes transpose. Since $Dh(x^*)$ is Hurwitz, $P$ is well defined. It is also easy to check that $P$ is symmetric and positive definite. From \cite[Theorem 4.6, p.\ 136]{khalil1996nonlinear}, we further have that $P$ is the unique positive definite and symmetric matrix satisfying the Liapunov equation
\begin{equation}
\label{eqn:LiapunovEquationSol}
Dh(x^*)^\tr P + PDh(x^*) = - \mathbb{I}_d,
\end{equation}
where, as mentioned before, $\mathbb{I}_d$ is the $d-$dimensional identity matrix. Let
\begin{equation}
\label{eqn:Defn_Z}
Z(x) = Dh(x)^\tr P + P Dh(x).
\end{equation}
From \eqref{eqn:LiapunovEquationSol}, $Z(x^*) = - \mathbb{I}_d.$ Let $\mathcal{N}(x^*)$ be a convex neighbourhood of $x^*$ such that
\[
\|Z(x) - Z(x^*)\| = \|Z(x) + \mathbb{I}_d\| \leq \kappa  \; \forall x \in \mathcal{N}(x^*),
\]
where $\kappa$ is as defined below \eqref{eqn:Defn_LambdaMinx*}. The existence of $\mathcal{N}(x^*)$ is guaranteed since $Z$ is continuous. The latter follows due to $\pmb{A_1}$ which ensures that $Dh$ is continuous.

Fix $r^\prime$ such that $0 < r^\prime < r$ and $V^{r^\prime} \subseteq \mathcal{N}(x^*).$
Fix $s,$ $u_0,$ and $u_1$ as prescribed in \textbf{Claim (i)} with $r^\prime$ as defined above. For notational convenience, let
\[
x_0(t) \equiv x(t, s, u_0)
\]
and
\[
x_1(t) \equiv x(t, s, u_1).
\]
Also let
\begin{equation}
\label{eqn:Defn_calV}
\mathcal{V}(t) = [x_0(t) - x_1(t)]^\tr P [x_0(t) - x_1(t)], \; t \geq s.
\end{equation}
Observe that since $P$ is positive definite, $\mathcal{V}(t) \geq 0$ for all $t \geq s.$  Differentiating with respect to $t$ and using the fact that
\[
\dot{x}_i(t) = h(x_i(t)),
\]
it is easy to see that
\begin{multline*}
\dot{\mathcal{V}}(t) = [h(x_0(t)) - h(x_1(t))]^\tr P [x_0(t) - x_1(t)] \\ + \; [x_0(t) - x_1(t)]^\tr P [h(x_0(t)) - h(x_1(t))].
\end{multline*}
By the mean value theorem,
\[
h(x_0(t)) - h(x_1(t)) = \left[\int_{0}^{1} Dh(x_1(t) + \tau [x_0(t) - x_1(t)]) d\tau \right] [x_0(t) - x_1(t)].
\]
Hence
\[
\dot{\mathcal{V}}(t) = [x_0(t) - x_1(t)]^\tr \left[\int_{0}^{1}Z(x_1(t) + \tau [x_0(t) - x_1(t)] ) d\tau \right] [x_0(t) - x_1(t)],
\]
where $Z$ is as in \eqref{eqn:Defn_Z}.

Since $V$ is a Liapunov function and $u_i \in V^{r^\prime},$ $x_i(t) \in V^{r^\prime} \subseteq \mathcal{N}(x^*)$ for all $t \geq s.$ Further since $\mathcal{N}(x^*)$ is convex, $x_1(t) + \tau [x_0(t) - x_1(t)] \in \mathcal{N}(x^*)$ for all $t \geq s$ and $\tau \in [0,1].$ By definition of $\mathcal{N}(x^*),$ for all $t \geq s$ and $\tau \in [0,1],$
\[
\|Z(x_1(t) + \tau [x_0(t) - x_1(t)]) + \mathbb{I}_d\| \leq \kappa.
\]
Hence by adding and subtracting $\mathbb{I}_d$ to the integrand in the relation concerning $\dot{\mathcal{V}}(t)$ above, it follows that
\[
\dot{\mathcal{V}}(t) \leq - (1 - \kappa) \|x_0(t) - x_1(t)\|^2.
\]
By definition of $\mathcal{V}(t)$ in \eqref{eqn:Defn_calV}, also note that
\[
\mathcal{V}(t) \leq \|P\| \; \|x_0(t) - x_1(t)\|^2.
\]
Combining the above two relations, we get
\[
\dot{\mathcal{V}}(t) \leq - \frac{1 - \kappa }{\|P\|} \mathcal{V}(t).
\]
But using \eqref{eqn:Defn_P} and \eqref{eqn:BoundNormDhx*}, note that
\begin{equation}
\label{eqn:PBound}
\|P\| \leq \frac{\tilde{K}^2}{2\lambda^\prime}.
\end{equation}
Hence using \eqref{eqn:Defn_lambda}, we have
\[
\dot{\mathcal{V}}(t) \leq - 2\left(\frac{1 - \kappa}{\tilde{K}^2}\right) \lambda^\prime \mathcal{V}(t) = - 2\lambda \mathcal{V}(t)
\]
and consequently, by integrating from $s$ to $t,$
\[
\mathcal{V}(t) \leq \mathcal{V}(s) e^{-2 \lambda(t - s)}.
\]

Since from \eqref{eqn:Defn_calV},
\[
\mathcal{V}(t) \geq \lambda_{\min}(P) \|x_0(t) - x_1(t)\|^2,
\]
and
\[
\mathcal{V}(s) \leq \|P\| \; \|x_0(s) - x_1(s)\|^2 =  \|P\| \; \|u_0 - u_1\|^2,
\]
it follows that
\begin{eqnarray*}
& & \|x(t, s, u_0) - x(t, s, u_1)\| \\
& = & \|x_0(t) - x_1(t)\| \\
& \leq & \sqrt{\frac{\mathcal{V}(t)}{\lambda_{\min}(P)}}\\
& \leq & \sqrt{\frac{\mathcal{V}(s) e^{-2\lambda (t - s)}}{\lambda_{\min}(P)}}\\
& \leq & K_1^\prime \|u_0 - u_1\| e^{-\lambda (t - s)},
\end{eqnarray*}
where $K_1^\prime := \sqrt{\frac{\|P\|}{\lambda_{\min}(P)}}.$ This proves \textbf{Claim (i)}, as desired.

We now proceed to prove the actual lemma. Pick arbitrary $u_0, u_1 \in V^r$ and $s \geq 0.$ Observe that
\[
x(t, s, u_i) = u_i + \int_{s}^{t} h(x(\tau, s, u_i)) \; d \tau.
\]
Hence we have
\[
\|x(t, s, u_0) - x(t, s, u_1)\| \leq \|u_0 - u_1\| + \int_{s}^{t} \|h(x(\tau, s, u_0)) - h(x(\tau, s, u_1))\|  \; d \tau.
\]
Since $u_i \in V^r$ and $V$ is a Liapunov function, $x(t, s, u_i) \in V^r$ for each $t \geq s.$ Hence, invoking the Lipschitz continuity of $h$ over $V^r,$ it follows that
\[
\|x(t, s, u_0) - x(t, s, u_1)\| \leq \|u_0 - u_1\| + L_h\int_{s}^{t} \|x(\tau, s, u_0) - x(\tau, s, u_1)\|  \; d \tau.
\]
Using Gronwall inequality \cite[Corollary 1.1]{bainov1992integral} on this, we get
\begin{equation}
\label{eqn:BoundCloseSolGronwall}
\|x(t, s, u_0) - x(t, s, u_1)\| \leq \|u_0 - u_1\| e^{L_h (t - s)}
\end{equation}
for any $t \geq s.$

Let $r^\prime$ be as in \textbf{Claim (i)} and let
\[
\mathscr{T}:= \frac{r - r^{\prime} }{\inf_{x \in V^r \backslash V^{r^\prime}}|\nabla V(x) \cdot h(x)|}.
\]
As $V$ is a Liapunov function, $\inf_{x \in V^r \backslash V^{r^\prime}}|\nabla V(x) \cdot h(x)| > 0.$ Since $\dot{V}(x(t)) = \nabla V(x(t)) \cdot h(x(t)),$ $\mathscr{T}$ is an upper bound on the time taken for a solution of \eqref{eqn:LimitingODE} starting from any point in $V^r$ to reach $V^{r^\prime}.$ That is, $x(s + \mathscr{T}, s, u_i) \in V^{r^\prime}$ whatever be the values of $s \geq 0$ and $u_i \in V^r.$ Combining this with \textbf{Claim (i)} above, it follows that for all $t \geq s + \mathscr{T},$
\begin{eqnarray*}
& & \|x(t, s, u_0) - x(t, s, u_1)\| \\
& \leq & K_1^\prime \|x(s + \mathscr{T}, s, u_0) - x(s + \mathscr{T}, s, u_1)\| \; e^{-\lambda (t - s - \mathscr{T})}.
\end{eqnarray*}
From \eqref{eqn:BoundCloseSolGronwall},
\[
\|x(s + \mathscr{T}, s, u_0) - x(s + \mathscr{T}, s, u_1)\| \leq \|u_0 - u_1\| e^{L_h \mathscr{T}}.
\]
Combining the above two, it follows that for $t \geq s + \mathscr{T},$
\[
\|x(t, s, u_0) - x(t, s, u_1)\| \leq K_1^\prime e^{L_h \mathscr{T}} \|u_0 - u_1\| e^{-\lambda (t - s - \mathscr{T})}.
\]
Hence for suitable $K_1 \geq 0,$ we have
\[
\|x(t, s, u_0) - x(t, s, u_1)\| \leq K_1 \|u_0 - u_1\| e^{-\lambda (t - s)}
\]
for all $t \geq s.$ This proves the desired result.
\end{proof}

\begin{lemma}
\label{lem:BoundDh}
Let $u_0 \in V^r$ and $s \geq 0$ be arbitrary. Then for any $t \geq s,$
\[
\int_{s}^{t}\|Dh(x(\tau, s, u_0)) - Dh(x^*)\| \; d\tau \leq K_2,
\]
where $K_2 \geq 0$ is some constant.
\end{lemma}
\begin{proof}
Recall that $Dh$ is Lipschitz continuous over the compact set $V^r$ with Lipschitz constant $L_D.$ Separately, since $V$ is a Liapunov function and $u_0 \in V^r,$ we have $x(\tau, s, u_0) \in V^r$ for all $\tau \geq s.$ Hence it follows that
\begin{eqnarray*}
& & \int_{s}^{t}\|Dh(x(\tau, s, u_0) - Dh(x^*)\| \\
& \leq & L_D \int_{s}^{t} \|x(\tau, s, u_0) - x^*\| d\tau\\
& \leq & L_D K_1 \int_{s}^{t} \|u_0 - x^*\| e^{-\lambda (\tau - s)} d \tau \\
& \leq & L_D K_1 R \int_{s}^{t} e^{-\lambda (\tau - s)} d \tau,
\end{eqnarray*}
where the second relation follows from Lemma~\ref{lem:CloseSolutionsBound} on substituting $u_1 = x^*,$ while the truth of the last one can be seen using \eqref{eqn:Defn_R}. Since
\[
\int_{s}^{t} e^{-\lambda (\tau - s)} d \tau \leq \int_{s}^{\infty} e^{-\lambda (\tau - s)} d \tau = \frac{1}{\lambda},
\]
the desired result is now easy to see.
\end{proof}

\begin{lemma}
\label{lem:BoundPhi}
Let $u_0 \in V^r$ and $s \geq 0$ be arbitrary. Let $\Phi(t, s, u_0),$  $t \geq s,$ be as defined above \eqref{eqn:VariationalEquation}. Then for $t \geq s,$
\[
\|\Phi(t, s, u_0)\| \leq K_3 e^{-\lambda (t - s)},
\]
where $K_3 \geq 0$ is some constant.
\end{lemma}
\begin{proof}
Observe that \eqref{eqn:VariationalEquation} can be written as
\[
\dot{y}(t) = Dh(x^*) y(t) + \left[Dh(x(t, s, u_0)) - Dh(x^*)\right] y(t)
\]
which can be thought of as a perturbation of \eqref{eqn:LinearizedSystem}. Hence, using the  variation of constants formula or equivalently the Alekseev formula (column by column), we get
\begin{multline}
\label{eqn:FundSolVarEqn}
\Phi(t, s, u_0) = e^{Dh(x^*) (t - s)} \\ +  \; \int_{s}^{t}e^{Dh(x^*) (t - \tau)} \left[Dh(x(\tau, s, u_0)) - Dh(x^*)\right]\Phi(\tau, s, u_0) d\tau.
\end{multline}
By \eqref{eqn:BoundNormDhx*},
\[
\|e^{Dh(x^*) (t - \tau)}\| \leq \tilde{K}e^{-\lambda^\prime(t - \tau)} \leq \tilde{K}e^{-\lambda (t - \tau)}, \; s \leq \tau \leq t,
\]
where $\lambda$ is as in \eqref{eqn:Defn_lambda}. Hence by taking spectral norm on both sides of \eqref{eqn:FundSolVarEqn}, we have
\begin{multline*}
\|\Phi(t, s, u_0)\| \leq \tilde{K} e^{-\lambda (t - s)} \\ +  \; \tilde{K} \int_{s}^{t}e^{-\lambda(t - \tau)} \|Dh(x(\tau, s, u_0)) - Dh(x^*)\| \; \|\Phi(\tau, s, u_0)\| d\tau.
\end{multline*}
Using Gronwall inequality \cite[Corollary 1.1]{bainov1992integral} on this, we get
\[
\|\Phi(t, s, u_0)\| \leq \tilde{K}\left(e^{-\lambda (t - s) + \tilde{K}\int_{s}^{t} \|Dh(x(\tau, s, u_0)) - Dh(x^*)\| d\tau}\right).
\]
By Lemma~\ref{lem:BoundDh}, the desired result follows.
\end{proof}

\begin{lemma}
\label{lem:BoundPhiDiff}
Let $u_0, u_1 \in V^r$ and $s \geq 0$ be arbitrary. Then for $t \geq s,$
\[
\|\Phi(t, s, u_0) - \Phi(t, s, u_1)\| \leq K_4 e^{-\lambda(t - s)} \|u_0 - u_1\|,
\]
where $\Phi(t, s, u_0)$ and $\Phi(t, s, u_1)$ are as defined above \eqref{eqn:VariationalEquation} and $K_4 \geq 0$ is some constant.
\end{lemma}
\begin{proof}
Recall from \eqref{eqn:VariationalEquation} that $\Phi(t, s, u_0)$ is the fundamental matrix of the ODE
\begin{equation}
\label{eqn:unperVarODE}
\dot{y}_0(t) = Dh(x(t, s, u_0)) y_0(t), \; t \geq s,
\end{equation}
while $\Phi(t, s, u_1)$ is the fundamental matrix of the ODE
\begin{equation}
\label{eqn:perVarODE}
\dot{y}_1(t) = Dh(x(t, s, u_1)) y_1(t), \; t \geq s.
\end{equation}
For $t \geq s^\prime \geq s,$ let $\Psi_0(t, s^\prime)$ denote the fundamental matrix of  \eqref{eqn:unperVarODE} satisfying
\[
\Psi_0(s^\prime, s^\prime) = \mathbb{I}_d.
\]
Similarly define $\Psi_1(t, s^\prime)$ with respect to \eqref{eqn:perVarODE}. Treating \eqref{eqn:perVarODE} as a perturbation of \eqref{eqn:unperVarODE}, it follows by using  variation of constants formula or equivalently Alekseev's formula (column by column) that
\begin{eqnarray*}
& & \Phi(t, s, u_1) - \Phi(t, s, u_0)  \\
& = & \Psi_1(t, s) - \Psi_0(t, s)\\
& = & \int_{s}^{t} \Psi_0(t, \tau) [Dh(x(\tau, s, u_1)) - Dh(x(\tau, s, u_0))]\Psi_1(\tau, s) d \tau.
\end{eqnarray*}

Since $u_0, u_1 \in V^r,$ it follows by arguing as in Lemma~\ref{lem:BoundPhi} that
\[
\|\Psi_0(t, \tau)\| \leq K_3 e^{-\lambda(t - \tau)}
\]
and
\[
\|\Psi_1(\tau, s)\| \leq K_3 e^{-\lambda(\tau - s)}
\]
Also recall that $Dh$ is Lipschitz continuous on $V^r$ with Lipschitz constant $L_D.$    Hence we have
\[
\|Dh(x(\tau, s, u_1)) - Dh(x(\tau, s, u_0))\| \leq L_D \|x(\tau, s, u_1) - x(\tau, s, u_0)\|.
\]
Putting all the above relations together, it follows that there exists some constant $K_4^\prime \geq 0$ such that
\[
\| \Phi(t, s, u_1) - \Phi(t, s, u_0) \| \leq K_4^\prime e^{-\lambda(t - s)} \int_{s}^{t} \|x(\tau, s, u_1) - x(\tau, s, u_0)\| d\tau.
\]
Using Lemma~\ref{lem:CloseSolutionsBound}, the desired result is now easy to see.
\end{proof}

\begin{lemma}
\label{lem:IntSquareExpTerm_a_k}
Let $k, n$ with $n_0 \leq k < k + 1 \leq n$ be arbitrary. Then there exists a constant $K_5 \geq 0$ such that, on the event $G_n,$
\[
\int_{t_{k}}^{t_{k + 1}} e^{-\lambda(t_n - s)} \|\bar{x}(s) - \bar{x}(t_k)\| ds \leq K_5 \left[1 + \|M_{k + 1}\| \right] e^{-\lambda(t_n - t_{k + 1})} a_k^2.
\]
\end{lemma}
\begin{proof}
From \eqref{eqn:LinearInterpolation}, note that
\[
\|\bar{x}(s) - \bar{x}(t_k)\| = \frac{(s - t_k)}{a_k} \|\bar{x}(t_{k + 1}) - \bar{x}(t_k)\| \leq (s - t_{k}) [\|h(\bar{x}(t_k))\| + \|M_{k + 1}\|],
\]
where the last inequality is due to \eqref{eqn:SAIterates}. On $G_n,$ and since $n_0\leq k \leq n - 1,$ note that $\bar{x}(t_k) \in V^r.$ Combining this with the fact that $h(x^*) = 0$ and $h$ is Lipschitz over $V^r,$  it follows that, on $G_n,$
\[
\|h(\bar{x}(t_k))\| = \|h(\bar{x}(t_k)) - h(x^*)\| \leq L_h R.
\]
Also note that
\[
\int_{t_k}^{t_{k + 1}}(s - t_k) e^{-\lambda(t_n - s)}ds \leq e^{-\lambda(t_n - t_{k + 1})} a_k^2.
\]
Combining the above relations, the desired result is easy to see.
\end{proof}

In the next two results, we respectively obtain bounds on $W_n$ and $\tilde{S}_n - S_n,$ where $W_n, \tilde{S}_n,$ and $S_n$ are as in \eqref{eqn:Defn_Wn}, \eqref{eqn:Defn_tildeSn}, and \eqref{eqn:Defn_Sn}.

\begin{lemma}
\label{lem:BoundW}
Let $n \geq n_0$ be arbitrary. Then on $G_n,$
\[
\|W_{n}\| \leq K_6\left[\sup_{n_0 \leq k \leq n - 1}a_k + \sup_{n_0 \leq k \leq n - 1} a_k \|M_{k + 1}\| \right],
\]
where $K_6 \geq 0$ is some constant.
\end{lemma}
\begin{proof}
Recall that $h$ is Lipschitz over $V^r$ with Lipschitz constant $L_h.$ Also, observe that
\[
\|W_n\| \leq \sum_{k = n_0}^{n - 1} \int_{t_k}^{t_{k + 1}} \|\Phi(t_n, s, \bar{x}(s))\| \; \|h(\bar{x}(t_k)) - h(\bar{x}(s))\| ds.
\]
Therefore, the following relations hold on the event $G_n.$ First, $\bar{x}(s) \in V^r$ for each $s \in [t_{n_0}, t_{n}].$  Hence,
\[
\|W_n\| \leq L_h \sum_{k = n_0}^{n - 1} \int_{t_k}^{t_{k + 1}} \|\Phi(t_n, s, \bar{x}(s))\| \; \|\bar{x}(t_k) - \bar{x}(s)\| ds.
\]
Using Lemma~\ref{lem:BoundPhi}, it now follows that
\[
\|W_n\| \leq L_h K_3 \sum_{k = n_0}^{n - 1} \int_{t_k}^{t_{k + 1}}e^{-\lambda (t_n - s)}\; \|\bar{x}(t_k) - \bar{x}(s)\| ds.
\]
Applying Lemma~\ref{lem:IntSquareExpTerm_a_k} to this gives
\[
\|W_n\| \leq L_h K_3 K_5\sum_{k = n_0}^{n - 1}[1 + \|M_{k + 1}\|] e^{-\lambda(t_n - t_{k + 1})} a_k^2.
\]
From this, it follows that there exists some constant $K_6^\prime \geq 0$ so that
\[
\|W_n\| \leq K_6^\prime\left[\sup_{n_0 \leq k \leq n - 1}a_k + \sup_{n_0 \leq k \leq n - 1} a_k \|M_{k + 1}\| \right] \sum_{k = n_0}^{n - 1} e^{-\lambda(t_n - t_{k + 1})} a_k.
\]
But observe that
\[
\sum_{k = n_0}^{n - 1} e^{-\lambda (t_n - t_{k + 1})}a_k  \leq \left[\sup_{k \geq 0} e^{\lambda a_k}\right] \int_{t_{n_0}}^{t_n} e^{-\lambda (t_n - s)} ds \leq \left[\sup_{k \geq 0} e^{\lambda a_k}\right]\frac{1}{\lambda} \leq \frac{e^{\lambda}}{\lambda},
\]
where the last inequality is due to \eqref{eqn:StepsizeAssumption3}. The desired result now follows.
\end{proof}

\begin{lemma}
\label{lem:BoundDiff_Sn_tildeS_n}
Let $n \geq n_0$ be arbitrary. Then on $G_{n},$
\[
\| \tilde{S}_n - S_n \|\leq  K_7\left[\sup_{n_0 \leq k \leq n - 1}a_k \|M_{k + 1}\| + \sup_{n_0 \leq k \leq n - 1} a_k \|M_{k + 1}\|^2\right],
\]
where $K_7 \geq 0$ is some constant.
\end{lemma}
\begin{proof}
Observe that
\[
\|\tilde{S}_n - S_n\| \leq \sum_{k = n_0}^{n - 1}\left[ \int_{t_k}^{t_{k + 1}} \|\Phi(t_n, s, \bar{x}(s)) - \Phi(t_n, s, \bar{x}(t_k))\| ds \right] \| M_{k + 1} \|.
\]
Hence, the following statements hold on the event $G_n.$ Clearly, $\bar{x}(s) \in V^r$ for each $s \in [t_{n_0}, t_{n}].$ Consequently, using Lemma~\ref{lem:BoundPhiDiff}, it follows that
\[
\|\tilde{S}_{n} - S_{n} \| \leq K_4 \sum_{k = n_0}^{n - 1} \left[\int_{t_k}^{t_{k + 1}} \|\bar{x}(s) - \bar{x}(t_k)\| e^{-\lambda (t_n - s)} ds \right] \|M_{k + 1}\|.
\]
Applying Lemma~\ref{lem:IntSquareExpTerm_a_k} to this shows that
\[
\|\tilde{S}_{n} - S_{n} \| \leq K_4 K_5 \sum_{k = n_0}^{n - 1} [1 + \|M_{k + 1}\|] \;  \|M_{k + 1}\| \; e^{-\lambda(t_n - t_{k + 1})} a_k^2
\]
Arguing now as in Lemma~\ref{lem:BoundW}, the desired result is easy to see.
\end{proof}

Assuming the event $G_n$ occurs, we now obtain upper bounds on $\|\bar{x}(t_n) - x(t_n, t_{n_0}, \bar{x}(t_{n_0}))\|$ and $\|\bar{x}(t_{n + 1}) - x(t_{n + 1}, t_{n_0}, \bar{x}(t_{n_0}))\|$ and use this to obtain bounds on $\rho_{n + 1}$ and $\rho_{n + 1}^*.$

\begin{lemma}
\label{lem:Bound_SA_ODE_t_n}
Let $n \geq n_0$ be arbitrary. Then on $G_n,$
\begin{multline*}
\|\bar{x}(t_n) - x(t_n, t_{n_0}, \bar{x}(t_{n_0}))\| \leq \\ K_8 \left[ \|S_n\| + \sup_{n_0 \leq k \leq n - 1} a_k + \sup_{n_0 \leq k \leq n - 1} a_{k} \|M_{k + 1}\|^2 \right],
\end{multline*}
where $K_{8} \geq 0$ is some constant.
\end{lemma}
\begin{proof}
From Theorem~\ref{thm:AlternativeExp_xnViaAlekseev}, we have
\[
\|\bar{x}(t_n) - x(t_n, t_{n_0}, \bar{x}(t_{n_0}))\| \leq \|W_n\| + \|S_{n}\| + \|\tilde{S}_{n} - S_{n}\|.
\]
Using Lemmas~\ref{lem:BoundW}, \ref{lem:BoundDiff_Sn_tildeS_n}, and the fact that $\|x\| \leq 1 + \| x \|^2,$ the desired result is easy to see.
\end{proof}

\begin{lemma}
\label{lem:Bound_SA_ODE_t_n+1}
Let $n \geq n_0$ be arbitrary. Then on $G_n,$
\begin{multline*}
\|\bar{x}(t_{n + 1}) - x(t_{n + 1}, t_{n_0}, \bar{x}(t_{n_0}))\| \leq \\
K_9 \left[\| S_n \| + \sup_{n_0 \leq k \leq n}a_k  + \sup_{n_0 \leq k \leq n}a_k \|M_{k + 1}\|^2 \right],
\end{multline*}
where $K_9 \geq 0$ is some constant.
\end{lemma}
\begin{proof}
Using \eqref{eqn:SAIterates} and
\[
x(t_{n + 1}, t_{n_0}, \bar{x}(t_{n_0})) = x(t_{n}, t_{n_0}, \bar{x}(t_{n_0})) + \int_{t_n}^{t_{n + 1}} h(x(s, t_{n_0}, \bar{x}(t_{n_0}))) ds,
\]
it follows from the triangle inequality that
\begin{multline*}
\|\bar{x}(t_{n + 1}) - x(t_{n + 1}, t_{n_0}, \bar{x}(t_{n_0}))\| \leq \|\bar{x}(t_{n}) - x(t_{n}, t_{n_0}, \bar{x}(t_{n_0}))\| \\
+ \; a_{n} \|M_{n + 1}\| + \int_{t_n}^{t_{n + 1}}\| h(\bar{x}(t_n)) - h(x(s, t_{n_0}, \bar{x}(t_{n_0})))\| ds.
\end{multline*}
But $h$ is Lipschitz over $V^r$ with Lipschitz constant $L_h.$ Also, on $G_n,$ $\bar{x}(t_n)$ and $x(s, t_{n_0}, \bar{x}(t_{n_0})),$ $s \geq t_{n_0},$ lie in $V^r.$ Hence it follows using \eqref{eqn:Defn_R} that
\[
\int_{t_n}^{t_{n + 1}}\| h(\bar{x}(t_n)) - h(x(s, t_{n_0}, \bar{x}(t_{n_0})))\| ds \leq 2 L_h R a_{n}.
\]
Substituting this in the above relation and using Lemma~\ref{lem:Bound_SA_ODE_t_n}, the desired result is easy to see.
\end{proof}

\begin{lemma}
\label{lem:Bound_rho_n+1}
Let $n \geq n_0$ be arbitrary. Then on $G_n,$
\[
\rho_{n + 1} \leq K_{10} \left[\| S_n \| + \sup_{n_0 \leq k \leq n}a_k  + \sup_{n_0 \leq k \leq n}a_k \|M_{k + 1}\|^2 \right],
\]
where $K_{10} \geq 0$ is some constant.
\end{lemma}
\begin{proof}
Fix $t \in [t_{n}, t_{n + 1}].$ Then there exists some $\pi \in [0,1]$ such that
\[
\bar{x}(t) = (1 - \pi) \bar{x}(t_n) + \pi \bar{x}(t_{n + 1}).
\]
Hence
\begin{multline*}
\|\bar{x}(t) - x(t, t_{n_0}, \bar{x}(t_{n_0}))\| \leq (1 - \pi) \|\bar{x}(t_n) - x(t, t_{n_0}, \bar{x}(t_{n_0})) \| \\
+ \; \pi \|\bar{x}(t_{n + 1}) - x(t, t_{n_0}, \bar{x}(t_{n_0}))\|.
\end{multline*}
Since
\[
x(t, t_{n_0}, \bar{x}(t_{n_0})) = x(t_n, t_{n_0}, \bar{x}(t_{n_0})) + \int_{t_{n}}^{t} h(x(s, t_{n_0}, \bar{x}(t_{n_0}))) ds
\]
and
\[
x(t_{n + 1}, t_{n_0}, \bar{x}(t_{n_0})) = x(t, t_{n_0}, \bar{x}(t_{n_0})) + \int_{t}^{t_{n + 1}} h(x(s, t_{n_0}, \bar{x}(t_{n_0}))) ds,
\]
we have
\begin{multline}
\label{eqn:Int_Bound_rho_n+1}
\|\bar{x}(t) - x(t, t_{n_0}, \bar{x}(t_{n_0}))\| \leq (1 - \pi)\|\bar{x}(t_n) - x(t_n, t_{n_0}, \bar{x}(t_{n_0}))\| \\
+ \; \pi \|\bar{x}(t_{n + 1}) - x(t_{n + 1}, t_{n_0}, \bar{x}(t_{n_0})) \| +   \int_{t_n}^{t_{n + 1}} \|h(x(s, t_{n_0}, \bar{x}(t_{n_0})))\|ds.
\end{multline}
But $h(x^*) = 0$ ensures
\[
\int_{t_n}^{t_{n + 1}} \|h(x(s, t_{n_0}, \bar{x}(t_{n_0})))\| ds = \int_{t_n}^{t_{n + 1}} \|h(x(s, t_{n_0}, \bar{x}(t_{n_0}))) - h(x^*)\| ds.
\]
Hence arguing as in the proof of Lemma~\ref{lem:Bound_SA_ODE_t_n+1}, it follows that on $G_n,$
\[
\int_{t_n}^{t_{n + 1}} \| h(x(s, t_{n_0}, \bar{x}(t_{n_0}))) \| ds \leq L_h R a_n.
\]
Substituting  this in \eqref{eqn:Int_Bound_rho_n+1} and making use of Lemmas~\ref{lem:Bound_SA_ODE_t_n} and \ref{lem:Bound_SA_ODE_t_n+1}, the desired result is easy to see.
\end{proof}

\begin{lemma}
\label{lem:Bound_rho_n+1^*}
Let $n \geq n_0$ be arbitrary. Then on $G_n,$
\[
\rho_{n + 1}^* \leq K_{11}\left[\|S_n\| + \sup_{n_0 \leq k \leq n}a_k  + \sup_{n_0 \leq k \leq n}a_k \|M_{k + 1}\|^2 + e^{-\lambda (t_{n} - t_{n_0})}\right],
\]
where $K_{11} \geq 0$ is some constant.
\end{lemma}
\begin{proof}
For $t \in [t_n, t_{n + 1}],$
\[
\|\bar{x}(t) - x^*\| \leq \|\bar{x}(t) - x(t, t_{n_0}, \bar{x}(t_{n_0}))\| + \|x(t, t_{n_0}, \bar{x}(t_{n_0})) - x^*\|.
\]
From Lemma~\ref{lem:CloseSolutionsBound}, it follow that, on $G_n,$
\[
\|x(t, t_{n_0}, \bar{x}(t_{n_0})) - x^*\| \leq K_1 \| \bar{x}(t_{n_0}) - x^*\| e^{-\lambda (t- t_{n_0})} \leq K_1 R e^{-\lambda (t - t_{n_0})}.
\]
Hence
\[
\rho_{n + 1}^* \leq \rho_{n + 1} + K_1 R \sup_{t \in [t_n, t_{n + 1}]} e^{-\lambda (t - t_{n_0})}.
\]
Using Lemma~\ref{lem:Bound_rho_n+1}, the desired result is easy to see.
\end{proof}

Let $K := \max\{K_{10}, K_{11}\}.$ The following result is then straightforward.
\begin{theorem}
\label{thm:Bound_rho_n+1_rho_n+1^*}
Let $n \geq n_0$ be arbitrary. Then on $G_n,$
\[
\rho_{n + 1} \leq K\left[\|S_n\| + \sup_{n_0 \leq k \leq n}a_k  + \sup_{n_0 \leq k \leq n}a_k \|M_{k + 1}\|^2 \right],
\]
and
\[
\rho_{n + 1}^* \leq K\left[\|S_n\| + \sup_{n_0 \leq k \leq n}a_k  + \sup_{n_0 \leq k \leq n}a_k \|M_{k + 1}\|^2 + e^{-\lambda (t_{n} - t_{n_0})}\right],
\]
where $K \geq 0$ is some constant.
\end{theorem}

\section{Proof of Theorem~\ref{thm:MainResult}}
\label{sec:MainProof}

Our first result here gives an upper bound for the probability expression on RHS of \eqref{eqn:IntConcBound} in terms of $\{\|S_n\|\}$ and $\{a_{n} \|M_{n + 1}\|^2\}.$

\begin{theorem}
\label{thm:ConcBoundInSn_Mn+1}
Let $\bar{x}(t)$ be as in \eqref{eqn:LinearInterpolation}, $K$ be as defined in Theorem~\ref{thm:Bound_rho_n+1_rho_n+1^*}, $n_1$ be as in \eqref{eqn:TReln1}, and $\epsilon$ be as in Theorem~\ref{thm:MainResult}. Let $N$ be such that $a_{n} \leq \epsilon/(4K)$ for all $n \geq N,$ and $T$ be such that $e^{-\lambda T} \leq \epsilon/(4K).$ Then for any $n_0 \geq N,$
\begin{multline}
\label{eqn:TwoTermSplit}
\Pr\left\{\bigcup_{n = n_0}^{n_1} \{G_n,  \rho_{n + 1} > \epsilon\} \cup \bigcup_{n = n_1 + 1}^{\infty} \{G_n, \rho_{n + 1}^* > \epsilon\} \bigg| \bar{x}(t_{n_0}) \in B\right\} \leq \\
\sum_{n = n_0}^{\infty} \Pr\left\{G_n, \|S_n\| > \frac{\epsilon}{4K} \bigg| \bar{x}(t_{n_0}) \in B\right\} \\
+ \; \sum_{n = n_0}^{\infty} \Pr\left\{G_n, a_n\|M_{n + 1}\|^2 > \frac{\epsilon}{4K} \bigg| \bar{x}(t_{n_0}) \in B\right\}.
\end{multline}
\end{theorem}
\begin{proof}
From \eqref{eqn:TReln1}, it follows that $t_{n} \geq t_{n_0} + T$ for each $n \geq n_1 + 1.$ Since $e^{-\lambda T} \leq \epsilon/(4K),$ it follows that for each $n \geq n_1 + 1,$ $e^{-\lambda(t_n - t_{n_0})} \leq \epsilon/(4K).$ Combining this with the fact that $n_0 \geq N,$ it follows from Theorem~\ref{thm:Bound_rho_n+1_rho_n+1^*} that, for $n_0 \leq n \leq n_1,$
\begin{multline*}
\left\{G_n, \rho_{n + 1} > \epsilon \right\} \subseteq \left\{G_n, \|S_n\| > \frac{\epsilon}{4K}\right\} \\ \cup \left\{G_n,  \sup_{n_0 \leq k \leq n} a_{k} \|M_{k + 1}\|^2 >  \frac{\epsilon}{4K}\right\},
\end{multline*}
and, for $n \geq n_1 + 1,$
\begin{multline*}
\{G_n, \rho_{n + 1}^*  > \epsilon \} \subseteq \left\{G_n, \|S_n\| > \frac{\epsilon}{4K} \right\} \\ \cup \left\{G_n, \sup_{n_0 \leq k \leq n} a_{k} \|M_{k + 1}\|^2 > \frac{\epsilon}{4K}\right\}.
\end{multline*}
For $n_0 \leq k \leq n,$ note that $G_n \subseteq G_k$ and hence
\[
\left\{G_n, a_k \|M_{k + 1}\|^2 > \frac{\epsilon}{4K} \right\} \subseteq  \left\{G_k, a_k \|M_{k + 1}\|^2 > \frac{\epsilon}{4K} \right\}.
\]
Thus for $n \geq n_0,$
\[
\left\{G_n, \sup_{n_0 \leq k \leq n} a_{k} \|M_{k + 1}\|^2 > \frac{\epsilon}{4K}\right\} \subseteq \bigcup_{k = n_0}^{n} \left\{G_k, a_k \|M_{k + 1}\|^2 > \frac{\epsilon}{4K}\right\}.
\]
Putting the above relations together, we have
\begin{multline*}
\bigcup_{n = n_0}^{n_1} \{G_n,  \rho_{n + 1} > \epsilon\} \cup \bigcup_{n = n_1 + 1}^{\infty} \{G_n, \rho_{n + 1}^* > \epsilon\} \subseteq \\
\bigcup_{n = n_0}^{\infty} \left\{G_n, \|S_n\| > \frac{\epsilon}{4K} \right\} \cup \bigcup_{n =  n_0}^{\infty} \left\{G_n, a_n \|M_{n + 1}\|^2 > \frac{\epsilon}{4K} \right\}.
\end{multline*}
The desired result is now easy to see.
\end{proof}

We now sequentially      derive bounds for the two expressions on RHS of \eqref{eqn:TwoTermSplit}. Let $K_{12} := \sup_{x \in V^r} c_1(x)$ and $K_{13} := \inf_{x \in V^r} c_2(x)/(2\sqrt{K}).$ Since $V^r$ is a compact set, it follows that $K_{12}, K_{13} \in (0, \infty).$

\begin{theorem}
\label{thm:Bound_an_Mn+1}
Let $\bar{x}(t)$ be as in \eqref{eqn:LinearInterpolation}, $K$ be as in Theorem~\ref{thm:Bound_rho_n+1_rho_n+1^*}, $\epsilon$ be as in Theorem~\ref{thm:MainResult}, and $N$ be as in Theorem~\ref{thm:ConcBoundInSn_Mn+1}. Then for $n_0 \geq N,$
\begin{multline*}
\sum_{n = n_0}^{\infty} \Pr\left\{G_n, a_n\|M_{n + 1}\|^2 > \frac{\epsilon}{4K} \bigg| \bar{x}(t_{n_0}) \in B\right\} \leq \\
K_{12}  \sum_{n = n_0}^{\infty} \exp \left(- \frac{K_{13}\sqrt{\epsilon}}{\sqrt{a_n}}\right).
\end{multline*}
\end{theorem}
\begin{proof}
Observe that
\begin{eqnarray*}
& & \Pr\left\{G_n, a_n\|M_{n + 1}\|^2 > \frac{\epsilon}{4K} \bigg| \bar{x}(t_{n_0}) \in B\right\} \\
& \leq & \Pr\left\{ a_n \|M_{n + 1}\|^2 > \frac{\epsilon}{4K}  \bigg| G_n, \bar{x}(t_{n_0}) \in B\right\}\\
& = & \Pr\left\{\|M_{n + 1}\| > \frac{\sqrt{\epsilon}}{2 \sqrt{K}\sqrt{a_n}}  \bigg| G_n, \bar{x}(t_{n_0}) \in B\right\}\\
& \leq & K_{12} \exp\left(- \frac{K_{13} \sqrt{\epsilon}}{\sqrt{a_n}} \right),
\end{eqnarray*}
where the last inequality follows due to \eqref{eqn:NoiseAssumption2} and the fact that $\bar{x}(t_n) \in V^r$ on the event $G_n.$ This proves the desired result.
\end{proof}

\begin{theorem}
\label{thm:BoundSn}
Let $\bar{x}(t)$ be as in \eqref{eqn:LinearInterpolation}, $K$ be as in Theorem~\ref{thm:Bound_rho_n+1_rho_n+1^*}, $\epsilon$ and $\beta_n$ be as in Theorem~\ref{thm:MainResult}, $N$ be as in Theorem~\ref{thm:ConcBoundInSn_Mn+1}, and $S_n$ be as in \eqref{eqn:Defn_Sn}. Then for some constants $K_{14} \geq 0$ and $K_{15} > 0,$ the following relation holds:
\[
\sum_{n = n_0}^{\infty} \Pr\left\{G_n, \|S_n\| > \frac{\epsilon}{4K} \bigg| \bar{x}(t_{n_0}) \in B\right\} \leq K_{14}  \sum_{n = n_0}^{\infty} \exp \left(- \frac{K_{15} \min\{\epsilon, \epsilon^2\}}{\beta_n}\right).
\]
\end{theorem}

\begin{proof}
Let
\begin{equation}
\label{eqn:Defn_Alpha_k+1n}
\alpha_{k + 1, n} := \int_{t_k}^{t_{k + 1}} \Phi(t_n, s, \bar{x}(t_k)) ds.
\end{equation}
Then $S_n = \sum_{k = n_0}^{n - 1} \alpha_{k + 1, n}  M_{k + 1}.$ Since $G_{n_0} \supseteq \cdots \supseteq G_{n-1} \supseteq G_{n},$ we have
\begin{eqnarray*}
& & \Pr \left\{G_n, \|S_n\| > \frac{\epsilon}{4 K} \bigg| \bar{x}(t_{n_0}) \in B\right\} \\
& \leq & \Pr \left\{G_{n - 1},  \|S_n\| > \frac{\epsilon}{4 K} \bigg| \bar{x}(t_{n_0}) \in B\right\} \\
& = & \Pr \left\{ \|S_n\| > \frac{\epsilon}{4 K} \bigg|G_{n - 1}, \bar{x}(t_{n_0}) \in B\right\}\Pr \left\{G_{n - 1} | \bar{x}(t_{n_0}) \in B\right\}\\
& = & \Pr \left\{ \bigg\|\sum_{k = n_0}^{n - 1} \alpha_{k + 1, n} M_{k + 1} 1_{G_{k}}\bigg\| > \frac{\epsilon}{4 K} \bigg|G_{n - 1}, \bar{x}(t_{n_0}) \in B\right\}\\
& & \times \Pr \left\{G_{n - 1} | \bar{x}(t_{n_0}) \in B\right\}\\
&=& \Pr \left\{ \bigg\| \sum_{k = n_0}^{n - 1} \alpha_{k + 1, n} M_{k + 1} 1_{G_{k}}\bigg\| > \frac{\epsilon}{4 K} \bigg| \bar{x}(t_{n_0}) \in B\right\}.
\end{eqnarray*}
The last but one equality follows since $1_{G_{n_0}} = \cdots = 1_{G_{n - 1}} =  1$ on $G_{n - 1}.$ To prove the desired result, it thus suffices to show that there exist constants $K_{14} \geq 0$ and $K_{15} > 0$ so that the following relation holds:
\begin{multline*}
\Pr \left\{ \bigg\| \sum_{k = n_0}^{n - 1} \alpha_{k + 1, n} M_{k + 1} 1_{G_{k}}\bigg\|  > \frac{\epsilon}{4K} \bigg| \bar{x}(t_{n_0}) \in B\right\}  \leq  \\K_{14} \exp \left(- \frac{K_{15} \min\{\epsilon, \epsilon^2\}}{\beta_n}\right).
\end{multline*}

Since
\[
\mathbb{E}\left[\alpha_{k + 1, n} M_{k + 1} 1_{G_{k}} \bigg| \mathcal{F}_k\right] = 0, \; k \geq n_0,
\]
where $\mathcal{F}_k$ is as in $\pmb{A_3},$ it follows that
\[
\sum_{k = n_0}^{n - 1} \alpha_{k + 1, n} M_{k + 1} 1_{G_{k}}
\]
is a sum of martingale-differences. Hence the above two relations follow directly from a conditional variant of Theorem~\ref{thm:ConcMartingalesMultivariate} and the discussion in Remark~\ref{rem:AppropriateConcInequality} provided there exist constants $\delta, C, \gamma_1, \gamma_2 > 0$ so that
\begin{equation}
\label{eqn:CondionalExpCond}
\mathbb{E}\left[e^{\delta \|M_{k} 1_{G_{k - 1}}\|} \bigg| \mathcal{F}_{k - 1}\right] \leq C \; a.s.,  \; k \geq n_0 + 1,
\end{equation}
\begin{equation}
\label{eqn:BoundCoeffCond}
\sum_{k = n_0}^{n - 1} \|\alpha_{k + 1, n}\|1_{G_k} \leq \gamma_1,
\end{equation}
and
\begin{equation}
\label{eqn:SupCoeffCond}
\max_{n_0 \leq k \leq n - 1} \|\alpha_{k + 1,n} \| 1_{G_k}\leq \gamma_2 \beta_{n}.
\end{equation}

In the remainder of this proof, we establish \eqref{eqn:CondionalExpCond}, \eqref{eqn:BoundCoeffCond}, and \eqref{eqn:SupCoeffCond}. Pick arbitrary $F_{k - 1} \in \mathcal{F}_{k - 1}.$ Then observe that
\begin{eqnarray}
& & \mathbb{E}[e^{\delta \|M_{k} 1_{G_{k - 1}}\|}1_{F_{k - 1}}] \nonumber \\
& = & \mathbb{E}\left[e^{\delta \|M_{k}\|} \bigg|F_{k - 1} G_{k - 1}\right]  \Pr\{F_{k - 1} G_{k - 1}\} + \Pr\{F_{k - 1} G_{k - 1}^c\}\nonumber \\
& \leq & \left[\int_{0}^{\infty} \Pr\left\{e^{\delta \|M_k\|} > u \bigg|G_{k - 1} F_{k - 1}\right\} du \right] \Pr\{F_{k - 1}\} + \Pr\{F_{k - 1}\}. \label{eqn:CondExpBound}
\end{eqnarray}
But
\begin{eqnarray*}
& & \int_{0}^{\infty} \Pr\left\{e^{\delta \|M_k\|} > u \bigg|G_{k - 1} F_{k - 1}\right\} du\\
& \leq & e^{\delta \ul} + \int_{e^{\delta \ul}}^{\infty} \Pr\left\{e^{\delta \|M_k\|} > u \bigg|G_{k - 1} F_{k - 1}\right\} du\\
& = & e^{\delta \ul} + \int_{e^{\delta \ul}}^{\infty} \Pr\left\{\|M_k\| > \frac{\log u}{\delta} \bigg|G_{k - 1} F_{k - 1}\right\} du,
\end{eqnarray*}
where $\ul$ is as in $\pmb{A_3}.$

Also note that when $u \geq e^{\delta \ul},$ we have $\log u/ \delta \geq \ul.$ Therefore,
\begin{eqnarray*}
& & \Pr\left\{\|M_k\| > \frac{\log u}{\delta} \bigg|G_{k - 1} F_{k - 1}\right\} \\\
& = & \frac{\mathbb{E}\left[I\left[\|M_k\| > \frac{\log u}{\delta}\right] I[ G_{k - 1} F_{k - 1}]\right]}{\Pr\{G_{k - 1} F_{k - 1}\}} \\
& = & \frac{\mathbb{E}\left[\Pr\left\{\|M_k\| > \frac{\log u}{\delta}\big| \mathcal{F}_{k - 1}\right\} I[G_{k - 1} F_{k - 1}]\right]}{\Pr\{G_{k - 1} F_{k - 1}\}}\\
&\leq & K_{12} e^{-K_{13} \log u/\delta},
\end{eqnarray*}
where the last inequality follows from $\pmb{A_3}$ and the fact that $x_{k - 1} \in V^r$ on the event $G_{k - 1}.$

If we pick $\delta = K_{13}/2,$ it follows from the above two inequalities that
\[
\int_{0}^{\infty} \Pr\left\{e^{\delta \|M_k\|} > u \bigg|G_{k - 1} F_{k - 1}\right\} du \leq \exp[K_{13} \ul/2] + \frac{K_{12}}{\exp[K_{13}\ul/2]}.
\]
Substituting this in \eqref{eqn:CondExpBound}, it follows that for
\[
C = \exp[K_{13} \ul/2] + \frac{K_{12}}{\exp[K_{13}\ul/2]} + 1 \text{ and } \delta = K_{13}/2,
\]
\[
\mathbb{E}[e^{\delta \|M_{k} 1_{G_{k - 1}}\|} 1_{F_{k - 1}}] \leq C \Pr\{F_{k - 1}\}.
\]
Since $F_{k - 1} \in \mathcal{F}_{k - 1}$ was arbitrary, we have
\[
\mathbb{E}[e^{\delta \|M_{k} 1_{G_{k - 1}}\|} | \mathcal{F}_{k - 1}] \leq C \; a.s.
\]
This establishes \eqref{eqn:CondionalExpCond}.

Next note from Lemma~\ref{lem:BoundPhi} that, on $G_k,$
\[
\|\Phi(t_n, s, \bar{x}(t_k))\| \leq K_3 e^{-\lambda (t_n - s)}.
\]
Hence from \eqref{eqn:Defn_Alpha_k+1n}, as in the Proof of Lemma~\ref{lem:BoundW}, it follows that
\[
\sum_{k = n_0}^{n - 1} \|\alpha_{k + 1, n} \| 1_{G_k} \leq K_3\sum_{k = n_0}^{n - 1} e^{-\lambda(t_n - t_{k + 1})} a_k \leq \frac{K_3e^{\lambda}}{\lambda},
\]
and
\[
\max_{n_0 \leq k \leq n - 1} \|\alpha_{k + 1, n}\| 1_{G_k} \leq K_3\beta_n,
\]
as desired in \eqref{eqn:BoundCoeffCond} and \eqref{eqn:SupCoeffCond}. This completes the proof.
\end{proof}

\begin{proof}[Proof of Theorem~\ref{thm:MainResult}]
Let $g_1(\epsilon) = \log[4K/\epsilon] /\lambda$ and $g_2(\epsilon) = 4 K/ \epsilon.$
Then the desired result follows from \eqref{eqn:IntConcBound} and Theorems~\ref{thm:ConcBoundInSn_Mn+1}, \ref{thm:Bound_an_Mn+1}, and \ref{thm:BoundSn}.
\end{proof}

We end this section with a brief comment on how one may estimate the constant $K_1^\prime$ defined in the proof of Lemma~\ref{lem:CloseSolutionsBound}. This is a key constant since all other constants defined throughout Sections~\ref{sec:ErrorBound} and \ref{sec:MainProof} essentially depend on it. First, $\tilde{K}$ and $\lambda^\prime$ defined in \eqref{eqn:BoundNormDhx*} do depend on the prior knowledge of $x^*$ which is usually unavailable. One can, though, use a loose estimate based on the knowledge of $Dh$ in a neighborhood of $x^*,$ if available. Having chosen $\tilde{K}$ and $\lambda^\prime,$ an estimate of $\|P\|$ can then be easily found via \eqref{eqn:PBound}. When the matrix $Dh(x^*)$ is symmetric, one can be a bit more explicit. In that case, $P = -[Dh(x^*)]^{-1}/2$  and $\tilde{K}$ and $\lambda^\prime$ can be chosen as in Footnote~   \ref{fn:splCase}; consequently, $K_1^\prime$ is precisely the square root of the condition number of $Dh(x^*).$ It may be noted that even in absence of explicit constants, our concentration bound does provide useful information as `order' estimates in the spirit of sample complexity in machine learning \cite{vapnik2013nature}.

\section{Discussion}
\label{sec:Discussion}

Here we first look at the issue of obtaining unconditional convergence rates/concentration bounds, as opposed to ours which is conditioned on the iterate being in the domain of attraction of a given equilibrium. An unconditional estimate will be a product of our estimate times the probability that the conditioning event occurs, i.e., the domain of attraction is indeed reached (one might add a qualifier `after a specified time'); e.g., see  Proposition 7.5, \cite{benaim1999dynamics}. As already noted, the latter is strictly positive for any stable equilibrium under reasonable hypotheses; hence, the primary task is to find a good estimate thereof.

The simplest case is when the limiting ODE has a single globally asymptotically stable equilibrium. In a recent work \cite{dalal2018finite1}, we obtained unconditional convergence rates for the special case of TD(0) with linear function approximation; this is a popular algorithm in reinforcement learning. There the limiting ODE is linear and consequently has only one unique equilibrium. The key idea there is to first obtain a high probability bound on how far the TD(0) iterates can go when the stepsizes are initially large. Once the stepsizes become sufficiently small, analysis of the present work is invoked in order to show that the TD(0) iterates closely follow an appropriate solution of the ODE with high probability. Hence, combining ideas from \cite{dalal2018finite1} and this work, we believe that it may be possible to obtain unconditional convergence rates for nonlinear SA methods whose limiting ODE has a unique, global, asymptotically stable equilibrium, as in \cite{frikha2012concentration, fathi2013transport}, but without having to resort to the strong HL or HLS$_\alpha$ type assumptions.

For the case of multiple equilibria/attractors, one has to distinguish between two scenarios. First is the case when the equilibria are unknown and while one of them may be the most desirable, the a priori description of it does not allow us to say anything about its location. This is commonplace in engineering applications; a prime example being the stochastic gradient scheme for minimization which guarantees convergence only to a local minimum whereas the desired goal is the global minimum. One way to ensure the latter is to add extraneous slowly decreasing noise, which leads to the simulated annealing algorithm \cite{gelfand1991recursive}. For a non-gradient scheme a similar ploy may be expected to lead to the minimum of the so called Freidlin-Wentzell potential \cite{freidlin2014random}; to our knowledge this has been worked out so far only in discrete state space \cite{miclo1992fini} and compact Riemannian manifolds \cite{miclo1992compacte}.

The other possible scenario is where there may be some prior information about possible equilibria/attractors and we wish to reach  a most preferred one. This may be the case, e.g., in models arising in economics; in fact this was the original motivation for Arthur to look at lock-in probability. Then the issue is what aspect of the dynamics, given that it is a socioeconomic process and not an algorithm, is in our control. In other words, can we affect the probability to reach the domain of attraction of the desired equilibrium from the given starting point. A natural and commonplace situation is when the initial point is in the domain of attraction of an undesired equilibrium. Then the complement to our probability estimate (i.e., $1 -$ the estimate) is an upper bound on the probability of escape from it. The paths from the initial point to the desired set may traverse several such domains of attraction and the upper bound will then involve all such estimates, for all possible traversal sequences.  This is an interesting direction to pursue in future. A second issue then is to improve this probability \textit{if} we have any control over the dynamics, including the possibility of adding noise as described above. This is a more interesting class of problems with overtones of `stochastic resonance' \cite{herrmann2013stochastic}.

Other interesting directions to pursue are extensions to distributed asynchronous algorithms and more general noise models such as Markov noise.

We end by pointing at some recent papers that build upon the ideas discussed here, thereby illustrating the usefulness of this work. In  \cite{dalal2018finite2} and \cite{borkar2018concentration}, concentration bounds have been obtained for two-timescale SA; the first one deals with the linear case, while the second one handles the generic non-linear setup. Separately, \cite{kumar2018bounds} studies constant stepsize SA used to track a slowly moving target and provides bounds on the tracking error.

\appendix
\section{Concentration Inequality for a sum of Martingale-Differences}
\label{sec:AppConc}

Extending \cite[Theorem 1.1]{liu2009exponential} and building upon its proof technique, we obtain here a novel concentration result for a sum of martingale-differences, first for the univariate case (Theorem~\ref{thm:ConcMartingales}) and then for the multivariate case (Theorem~\ref{thm:ConcMartingalesMultivariate}).

\begin{theorem}
\label{thm:ConcMartingales}
Let $\{X_k\}$ be a real valued $\{\mathcal{F}_k\}-$adapted martingale-difference sequence. Assume that there exist $\delta, C > 0$ such that
\[
\mathbb{E}[e^{\delta |X_k|} | \mathcal{F}_{k - 1}] \leq C \; a.s.
\]
for all $k \geq 1.$ Let $S_n = \sum_{k = 1}^{n} \alpha_{k, n} \, X_k,$ where $\{\alpha_{k, n}\}$ are a.s. bounded previsible real valued random variables. That is, $\alpha_{k, n} \in \mathcal{F}_{k - 1}$ and there is a finite positive deterministic number, say $A_{k , n},$ such that $|\alpha_{k, n}| \leq A_{k, n}$ a.s. Suppose $\sum_{k = 1}^{n} A_{k, n} \leq \gamma_1$ and $\max_{1 \leq k \leq n} A_{k, n} \leq \gamma_2 \beta_n,$ where $\{\beta_n\}$ is some positive sequence and $\gamma_1, \gamma_2 > 0$ are constants that are independent of $n.$ Then there exists some constant $c > 0$ depending on $\delta,$ $C,\gamma_1, \gamma_2$ such that, for $\xi > 0,$
\[
\Pr\{|S_n| > \xi\} \leq
\begin{cases}
2\exp\left(-\frac{c \xi^2}{ \beta_{n}}\right), & \text{ if $\xi \in (0, \frac{C \gamma_1}{\delta}]$}, \\
2\exp\left(-\frac{c \xi}{\beta_{n} }\right), & \text{otherwise.}
\end{cases}
\]
\end{theorem}

We divide the proof into a series of lemmas.

\begin{lemma}
\label{lem:expSum}
Let $\alpha_{k, n}, S_n, \delta, X_k, \mathcal{F}_{k - 1}$ be as in Theorem~\ref{thm:ConcMartingales}. Suppose there exist functions $\{\ell_k : \mathbb{R}_{+} \to \mathbb{R}\}_{1 \leq k\leq n}$ such that for $\omega \geq 0$
\[
\mathbb{E}[\exp\left(\omega \, \alpha_{k, n} \, \delta X_k\right)|\mathcal{F}_{k - 1}] \leq e^{\ell_{k}(\omega)} \; a.s.
\]
Then
\[
\mathbb{E}[e^{\omega \delta S_n}] \leq e^{\sum_{k=1}^{n} \ell_k(\omega)}.
\]
\end{lemma}
\begin{proof}
This  follows from iterated conditioning.
\end{proof}

\begin{lemma}
\label{lem:expBound}
Let $\alpha$ be some bounded real valued random variable with $|\alpha| \leq A$ a.s. Let $X$ be another real valued random variable with $\mathbb{E}[\alpha X] = 0$ and $\mathbb{E}[e^{\delta |X|}] \leq C$ for some $\delta, C > 0.$ Then for all $0 < \omega < 1/A,$
\[
\mathbb{E}[e^{\omega \alpha \delta X}] \leq \exp\left[\frac{C A^2 \omega^2}{1 - A \omega}\right].
\]
\end{lemma}
\begin{proof}
Fix arbitrary $\omega$ such that $0 < \omega < 1/A.$ Since $\mathbb{E}[\alpha X] = 0,$
\begin{eqnarray*}
\mathbb{E}[e^{\omega \alpha \delta X}] & = & \sum_{k = 0}^{\infty} (\omega)^k \mathbb{E}\left[\frac{(\delta \alpha X)^k}{k!}\right] = 1 + \sum_{k = 2}^{\infty} (\omega)^k \mathbb{E}\left[\frac{(\delta \alpha X)^k}{k!}\right] \\
& \leq &  1 + \sum_{k = 2}^{\infty} (\omega A)^{k} \mathbb{E}[e^{\delta |X|}] \leq 1 + C \frac{A^2 \omega^2}{1 - A \omega} \leq \exp\left[\frac{C A^2 \omega^2}{1 - A \omega}\right]
\end{eqnarray*}
as desired.
\end{proof}

\begin{proof}[Proof of Theorem~\ref{thm:ConcMartingales}] Because $\alpha_{k, n}$ is previsible, note that
\[
\mathbb{E}[\alpha_{k, n} X_k|\mathcal{F}_{k - 1}] = 0.
\]
Let $0 < \omega < 1/(\gamma_2\beta_n).$ Then from a conditional variant of Lemma~\ref{lem:expBound}, it follows that
\[
\mathbb{E}[\exp\left(\omega \alpha_{k, n} \delta X_k\right)|\mathcal{F}_{k - 1}]  \leq \exp\left[\frac{CA_{k, n}^2 \omega^2}{1 - A_{k, n}\omega}\right] \leq \exp\left[\frac{C \gamma_2 \beta_{n} A_{k, n} \omega^2}{1 - \gamma_2 \beta_{n} \omega}\right] \; a.s.
\]
Hence, by Lemma~\ref{lem:expSum}, we obtain
\[
\mathbb{E}[e^{\omega \delta S_n}] \leq \exp\left[\frac{C\gamma_2 \omega^2 \beta_n \sum_{k = 1}^{n} A_{k, n}}{1 - \gamma_2\beta_{n} \omega}\right] .
\]
But $\sum_{k = 1}^{n} A_{k, n} \leq \gamma_1.$ Hence
\[
\mathbb{E}[e^{\omega \delta S_n}]  \leq \exp\left[\frac{C \gamma_1 \gamma_2 \omega^2 \beta_n}{1 - \gamma_2\beta_{n} \omega}\right].
\]
From this, it follows that
\[
\Pr\{S_{n} > \xi\} \leq \Pr\{e^{\omega\delta S_n} > e^{\omega \delta \xi}\} \leq \exp\left[-\left( \omega \delta \xi - \frac{C \gamma_1 \gamma_2 \omega^2 \beta_{n} }{1 - \gamma_2\beta_{n} \omega}\right)\right].
\]
Since this holds true for each $0 < \omega < 1/(\gamma_2\beta_{n}),$ we have
\[
\Pr\{S_{n} > \xi\} \leq \exp\left[-\sup_{\omega \in \left(0, \frac{1}{\gamma_2\beta_{n}}\right)} \left(\omega \delta \xi - \frac{C \gamma_1 \gamma_2 \omega^2 \beta_{n}}{1 - \gamma_2 \beta_{n} \omega}\right)\right].
\]
Now using \cite[Lemma 2.7]{liu2009exponential}, we get
\[
\Pr\{S_{n} > \xi\} \leq \exp\left[-\frac{\delta}{\gamma_2\beta_{n}} \left(\sqrt{\xi + \frac{C \gamma_1}{\delta}} - \sqrt{\frac{C\gamma_1}{\delta}}\right)^2 \right].
\]
Using the proof of \cite[(2.4)]{liu2009exponential}, it eventually follows that
\[
\Pr\{S_n > \xi\} \leq
\begin{cases}
\exp\left(-\frac{\delta^2 \xi^2}{C \gamma_1 \gamma_2 \beta_{n} (1 + \sqrt{2})^2}\right), & \text{ if $\xi \in (0, \frac{C \gamma_1}{\delta}],$} \\
\exp\left(-\frac{\delta \xi}{\gamma_2\beta_{n} (1 + \sqrt{2})^2}\right), & \text{otherwise.}
\end{cases}
\]
Similarly, one can show that
\[
\Pr\{S_n < -\xi\}  = \Pr\{-S_{n} > \xi\}\leq
\begin{cases}
\exp\left(-\frac{\delta^2 \xi^2}{C \gamma_1 \gamma_2 \beta_{n} (1 + \sqrt{2})^2}\right), & \text{if $\xi \in (0, \frac{C \gamma_1}{\delta}],$} \\
\exp\left(-\frac{\delta \xi}{\gamma_2 \beta_{n} (1 + \sqrt{2})^2}\right), & \text{otherwise.}
\end{cases}
\]
The desired result follows.
\end{proof}

The next result is a multivariate version of Theorem~\ref{thm:ConcMartingales}.
\begin{theorem}
\label{thm:ConcMartingalesMultivariate}
Let $S_n = \sum_{k = 1}^{n} \alpha_{k, n} \, X_k,$ where $\{X_k\}$ is a $\mathbb{R}^d$ valued  $\{\mathcal{F}_k\}-$adapted martingale-difference sequence and $\{\alpha_{k, n}\}$ is a sequence of a.s. bounded previsible  real valued $d \times d$ random matrices. That is, $\alpha_{k, n} \in \mathcal{F}_{k - 1}$ and there exists a finite number, say $A_{k, n},$ such that $\|\alpha_{k, n}\| \leq A_{k, n}$ a.s. Suppose that for some $\delta, C > 0$
\[
\mathbb{E}[e^{\delta \|X_k\|} | \mathcal{F}_{k - 1}] \leq C \; a.s.
\]
for each $k \geq 1.$ Further assume that $\sum_{k = 1}^{n} A_{k, n} \leq \gamma_1$ and $\max_{1 \leq k \leq n} A_{k, n} \leq \gamma_2 \beta_n,$ where $\{\beta_n\}$ is some positive sequence and $\gamma_1, \gamma_2 > 0$ are constants that are independent of $n.$ Then there exists some constant $c > 0$ depending on $\delta,$ $C,$ $\gamma_1, \gamma_2$ such that, for $\xi > 0,$
\begin{equation}
\label{eqn:ConcInequalityMultivariate}
\Pr\{\|S_n\| > \xi\} \leq
\begin{cases}
2d^2\exp\left(-\frac{c \xi^2}{d^3 \beta_{n}}\right), & \text{ if $\xi \in (0, \frac{C \gamma_1 d\sqrt{d}}{\delta}],$} \\
2d^2\exp\left(-\frac{c \xi}{d \sqrt{d} \beta_{n} }\right), & \text{otherwise.}
\end{cases}
\end{equation}
\end{theorem}
\begin{proof}
Let $\alpha_{k, n}^{ij}$ denote the $(i, j)$-th entry of the matrix $\alpha_{k, n}.$ Similarly, let $X_k^{j}$ denote the $j-$th entry of the vector $X_k.$ Then, it is easy to see that the $i-$th entry of the vector $S_n$ satisfies
\begin{equation}
\label{eqn:Defn_Sn^i}
S_{n}^{i} = \sum_{j = 1}^{d} \left[\sum_{k = 1}^{n} \alpha_{k, n}^{ij} X_k^{j} \right].
\end{equation}
Hence it follows that
\begin{eqnarray}
& & \Pr\{\|S_{n}\| > \xi\} \nonumber\\
& \leq & \sum_{i = 1}^{d} \Pr\left\{|S_{n}^i| > \frac{\xi}{\sqrt{d}}\right\} \nonumber\\
& \leq & \sum_{i = 1}^{d}\sum_{j = 1}^{d} \Pr\left\{\left|\sum_{k = 1}^{n} \alpha_{k, n}^{ij} X_k^j \right| > \frac{\xi}{d \sqrt{d}}\right\}. \label{eqn:ConcIneqIndividualSums}
\end{eqnarray}
Observe that, almost surely,
\[
\mathbb{E}[e^{\delta |X_k^j|}|\mathcal{F}_{k - 1}] \leq \mathbb{E}[e^{\delta \|X_k\|} | \mathcal{F}_{k - 1}] \leq C
\]
and
\[
\sum_{k = 1}^{n}| \alpha_{k, n}^{i, j}| \leq \sum_{k = 1}^{n} \|\alpha_{k, n}\| \leq \sum_{k = 1}^{n} A_{k,n} \leq \gamma_1.
\]
Also, $\max_{1 \leq k \leq n} |\alpha_{k, n}^{i, j} | \leq \max_{1 \leq k \leq n} \|\alpha_{k, n}\| \leq \max_{1 \leq k \leq n} A_{k,n} \leq \gamma_2 \beta_{n}$ a.s. Hence from Theorem~\ref{thm:ConcMartingales}, it follows that there exists some $c > 0$ depending on $C, \delta, \gamma_1, \gamma_2 $ such that
\[
\Pr\left\{\left|\sum_{k = 1}^{n} \alpha_{k, n}^{ij} X_k^j \right|  > \frac{\xi}{d \sqrt{d}}\right\} \leq
\begin{cases}
2\exp\left(-\frac{c \xi^2}{ d^3\beta_{n}}\right), & \text{ if $\xi \in (0, \frac{C \gamma_1 d \sqrt{d}}{\delta}],$} \\
2\exp\left(-\frac{c \xi}{d \sqrt{d} \beta_{n} }\right), & \text{otherwise.}
\end{cases}
\]
Using this in \eqref{eqn:ConcIneqIndividualSums}, the desired result is easy to see.
\end{proof}

\begin{remark}
\label{rem:AppropriateConcInequality}
To aid in the comparison with related literature in Section~\ref{sec:Overview}, we have highlighted in \eqref{eqn:ConcInequalityMultivariate} the dependence on $d.$ However, $d$ is often a constant (as it is assumed in this paper). In this situation, one can rephrase \eqref{eqn:ConcInequalityMultivariate} as
\begin{equation}
\Pr\{\|S_n\| > \xi\} \leq
\begin{cases}
c_1\exp\left(-\frac{c_2 \xi^2}{\beta_{n}}\right), & \text{ if $\xi \in (0, C^\prime],$} \\
c_1\exp\left(-\frac{c_2 \xi}{\beta_{n} }\right), & \text{otherwise}
\end{cases}
\end{equation}
for some suitably chosen constants $c_1, c_2, C^\prime > 0$  depending on $\delta, C, \gamma_1, \gamma_2$ and $d.$ In fact, by choosing $c_1, c_2$ appropriately, the above inequality can be rewritten as
\begin{equation}
\Pr\{\|S_n\| > \xi\} \leq
\begin{cases}
c_1\exp\left(-\frac{c_2 \xi^2}{\beta_{n}}\right) & \text{ if $\xi \in (0, 1]$} \\
c_1\exp\left(-\frac{c_2 \xi}{\beta_{n} }\right) & \text{otherwise}
\end{cases}.
\end{equation}
\end{remark}

\section{Order Estimates for Concentration Bound}
\label{sec:OrderEstimates}

Treating $n_0$ as the only variable, we first obtain here order estimates for the concentration bound given in Theorem~\ref{thm:MainResult} for the family of stepsizes $a_n = 1/(n + 1)^\mu,$ $\mu \in (0, 1].$

\begin{lemma}
\label{lem:anTermBd}
Let $C > 0$ be an arbitrary constant and let $a_n = 1/(n + 1)^{\mu},$ $\mu \in (0, 1].$ Then, as $n_0 \to \infty,$
\[
\sum_{n = n_0}^{\infty} e^{-C/\sqrt{a_n}} = O\left(n_0^{1 - \mu/2} e^{-C n_0 ^{\mu/2}}\right).
\]
\end{lemma}
\begin{proof}
Observe that
\[
\sum_{n = n_0}^{\infty} e^{-C/\sqrt{a_n}} \leq \int_{n_0 - 1}^{\infty}  e^{-C(s + 1)^{\mu/2}} ds = \int_{n_0}^{\infty} e^{-Cs^{\mu/2}} ds.
\]
Using l'H\^{o}pital's rule,
\[
\lim_{n_0 \to \infty}\!
\frac{\int_{n_0}^{\infty}  e^{-Cs^{\mu/2}} ds}{n_0^{1 - \mu/2} e^{-C n_0 ^{\mu/2}}} \! = \! \lim_{n_0 \to \infty} \! \frac{\frac{d}{d n_0} \left[\int_{n_0}^{\infty}  e^{-Cs^{\mu/2}} ds\right]}{\frac{d}{d n_0} \left[n_0^{1 - \mu/2} e^{-C n_0 ^{\mu/2}}\right]} = C^\prime,
\]
where $C^\prime \geq 0$ is some constant.  The desired result is now easy to see.
\end{proof}

\begin{lemma}
\label{lem:nStepSize}
Let $C > 0$ be an arbitrary constant, $a_n = 1/(n + 1),$ and $\beta_n$ be as in Theorem~\ref{thm:MainResult}. Then, as $n_0 \to \infty,$
\[
\sum_{n = n_0}^{\infty} e^{-C/\beta_n} =
\begin{cases}
O(e^{-Cn_0}) & \text{ if $\lambda > 1,$}\\
O(e^{-(C/2)n_0}) & \text{ if $\lambda \leq 1,$}
\end{cases}
\]
where $\lambda$ is as defined in \eqref{eqn:Defn_lambda}.
\end{lemma}
\begin{remark}
In Lemma~\ref{lem:nStepSize}, for the case $\lambda \leq 1,$ the below proof can be modified suitably to improve the estimate to $O(e^{-C^\prime n_0})$ for any $C^\prime < C.$
\end{remark}
\begin{proof}[Proof of Lemma~\ref{lem:nStepSize}]
First, consider the case $\lambda > 1.$  We claim here that, for all sufficiently large $n_0,$ and $k, n$ such that $n_0 \leq k \leq n-2,$
\[
\left[e^{-\lambda \sum_{i = k + 1}^{n - 1}a_i} \right] a_k \leq \left[e^{-\lambda \sum_{i = k + 2}^{n - 1}a_i} \right] a_{k + 1}.
\]
To prove this, it suffices to show that, for all sufficiently large $n_0,$ and $k \geq n_0,$
\begin{equation}
\label{eqn:MonotonePropertyStepsize}
e^{-\lambda a_{k + 1}} a_k \leq a_{k + 1}.
\end{equation}
But observe that
\[
\frac{a_{k}}{a_{k + 1}} = \frac{k + 2}{k + 1} \leq e^{1/(k + 1)}.
\]
Also, since $\lambda > 1,$ for all sufficiently large $n_0$ and $k \geq n_0,$ we have
\[
e^{1/(k + 1)} \leq e^{\lambda/(k + 2)}.
\]
Hence \eqref{eqn:MonotonePropertyStepsize} holds and our claim follows. From this, for all sufficiently large $n_0$ and $n \geq n_0,$ $\beta_{n} = a_{n - 1} = 1/n.$ Hence, for all sufficiently large $n_0,$
\[
\sum_{n = n_0}^{\infty}e^{-C/\beta_n}  \leq  \sum_{n = n_0}^{\infty} e^{-Cn} = O(e^{-C n_0})
\]
where the latter follows by treating the sum as a geometric series. The desired result now follows.

Next, consider the case $\lambda \leq 1.$ Clearly,
\[
\frac{a_{k + 1}}{a_k} = \frac{k + 1}{k + 2} \leq e^{-1/(k + 2)} \leq e^{-\lambda a_{k + 1}}.
\]
Hence, by arguing as above, it is easy to see that
\[
\left[e^{-\lambda \sum_{i = k + 1}^{n - 1}a_i} \right] a_k \geq  \left[e^{-\lambda \sum_{i = k + 2}^{n - 1}a_i} \right] a_{k + 1},
\]
for all $n_0,$ and  $k, n$ such that $n_0 \leq k \leq n-2.$ Fix $n_0$ and $n \geq n_0 + 1.$ Then, due to the previous relation,
\[
\beta_{n} = \left[e^{-\lambda \sum_{i = n_0 + 1}^{n - 1} a_i}\right] a_{n_0}.
\]
But
\[
\sum_{i = n_0 + 1}^{n - 1} a_i \geq \int_{n_0 + 1}^{n} \frac{1}{s + 1} ds =  \log\left(\frac{n + 1}{ n_0 + 2}\right).
\]
Using this, for sufficiently large $n_0,$ we have
\[
\beta_{n} \leq \frac{(n_0 + 2)^\lambda}{(n + 1)^\lambda} \frac{1}{n_0 + 1} \leq \frac{2}{(n + 1)^\lambda n_0^{1 - \lambda}}
\]
where the latter follows since $2^\lambda \leq 2.$ Hence, it follows that
\[
\sum_{n = n_0}^{\infty} e^{-C/\beta_n} \leq \sum_{n = n_0}^{\infty} e^{-(C/2) n_0^{1 - \lambda} (n + 1)^\lambda} \leq \int_{n_0}^{\infty} e^{-(C/2) n_0^{1 - \lambda} s^\lambda} ds.
\]
Separately, observe that under the transformation $s n_0^{(1 - \lambda)/\lambda} \to s$
\[
\int_{n_0}^{\infty} e^{-(C/2) n_0^{1 - \lambda} s^\lambda} ds = \frac{1}{n_0^{(1 - \lambda)/\lambda}}\int_{n_0^{1/\lambda}}^{\infty} e^{-(C/2) s^\lambda} ds = O(e^{-(C/2)n_0}),
\]
where the latter follows from l'H\^{o}pital's rule. Substituting this in the previous relation gives the desired result.
\end{proof}

\begin{lemma}
\label{lem:bnTermBd}
Let $C > 0$ be an arbitrary constant, $a_n = 1/(n + 1)^{\mu}$ with $\mu \in (0, 1),$ and $\beta_n$ be as in Theorem~\ref{thm:MainResult}. Then, as $n_0 \to \infty,$
\[
\sum_{n = n_0}^{\infty} \exp\left(- \frac{C}{\beta_n}\right) = O\left(n_0 ^{1 - \mu} e^{-C (n_0 - 1)^\mu} \right).
\]
\end{lemma}
\begin{proof}
As in the proof of the previous lemma, we first show that, for sufficiently large $n_0,$ and all $k \geq n_0,$
\begin{equation}
\label{eqn:stepSizeRel}
\frac{a_k}{a_{k + 1}} \leq e^{\lambda a_{k + 1}}.
\end{equation}
Observe that
\[
\frac{a_{k}}{a_{k + 1}} = \left(\frac{k + 2}{k + 1}\right)^\mu \leq e^{\mu/(k + 1)}.
\]
Since $\mu < 1,$ we have $\mu/(k + 1) \leq \lambda/ (k + 2)^\mu = \lambda a_{k + 1},$ for sufficiently large $n_0,$ and all $k \geq n_0.$ Hence \eqref{eqn:stepSizeRel} holds and therefore
\[
\beta_{n} = a_{n - 1} = \frac{1}{n^\mu}.
\]
Using this and l'H\^{o}pital's rule,
\[
\sum_{n = n_0}^{\infty} e^{-C/\beta_n} = \sum_{n = n_0}^{\infty} e^{-C n^\mu} \leq  O\left(n_0 ^{1 - \mu} e^{-C (n_0 - 1)^\mu} \right).
\]
This proves the desired result.
\end{proof}

\begin{proof}[Proof of Theorem~\ref{thm:OrderEst}]
Observe that the bound in Lemma~\ref{lem:anTermBd} dominates that in Lemmas~\ref{lem:nStepSize} and \ref{lem:bnTermBd}, respectively, for the cases $\mu = 1$ and $\mu \in (0, 1).$ The desired result is now easy to see.
\end{proof}

\section*{Acknowledgements}
The authors would like to thank the Bharti Centre for Communication, IIT Bombay where a portion of this work was completed.

\bibliographystyle{imsart-number}
\bibliography{SCA_Bib}

\end{document}